\documentclass[11pt]{amsart}

\usepackage{amssymb,latexsym}

\usepackage{enumerate}

\makeatletter

\@namedef{subjclassname@2010}{%

  \textup{2010} Mathematics Subject Classification}

\makeatother

\usepackage[english]{babel}
\usepackage{graphicx}
\usepackage{latexsym}
\usepackage{amsmath}
\usepackage{amsfonts}
\usepackage{amssymb, amsthm}
\usepackage{mathrsfs, psfrag}
\usepackage{stmaryrd}

\numberwithin{equation}{section}

\begin{document}

\baselineskip=17pt

\newtheorem{thm}{Theorem}
\newtheorem*{thma}{Theorem A}
\newtheorem*{thmb}{Theorem B}

\newtheorem{cor}[thm]{Corollary}
\newtheorem{counterexample}[thm]{Counterexample}
\newtheorem{conj}[thm]{Conjecture}
\newtheorem{lem}[thm]{Lemma}
\newtheorem{prop}[thm]{Proposition}

\newtheorem*{moore}{Moore's Theorem}
\newtheorem*{poincare}{Poincar\'e Classification Theorem}

\theoremstyle{definition}

\newtheorem{defn}{Definition}
\newtheorem{example}{Example}
\newtheorem{counter}{Counterexample}
\newtheorem{quest}{Question}
\newtheorem{notation}[thm]{Notation}
\newtheorem*{claim}{Claim}

\theoremstyle{open}
\newtheorem{open}{Open problem}
\newtheorem{rem}{Remark}

\newcommand{\PP}{{\mathbb P}}
\newcommand{\HH}{{\mathbb H}}
\newcommand{\C}{{\mathbb C}}
\newcommand{\M}{{\mathbb M}}
\newcommand{\N}{{\mathbb N}}
\newcommand{\Q}{{\mathbb Q}}
\newcommand{\R}{{\mathbb R}}
\newcommand{\T}{{\mathbb T}}
\newcommand{\V}{{\mathbb V}}
\newcommand{\U}{{\mathbb U}}
\newcommand{\X}{{\mathbb X}}
\newcommand{\Z}{{\mathbb Z}}
\newcommand{\D}{{\mathbb D}}
\newcommand{\B}{{\mathbb B}}
\newcommand{\s}{{\mathbb S}}

\newcommand{\A}{{\mathcal A}}
\newcommand{\BB}{{\mathcal B}}
\newcommand{\UU}{{\mathcal U}}
\newcommand{\CC}{{\mathcal C}}
\newcommand{\DD}{{\mathcal D}}
\newcommand{\MM}{{\mathcal M}}
\newcommand{\SSS}{{\mathcal S}}
\newcommand{\OO}{{\mathcal O}}
\newcommand{\QQ}{{\mathcal Q}}
\newcommand{\LL}{{\mathcal L}}

\newcommand{\KK}{{\mathcal K}}
\newcommand{\WW}{{\mathcal W}}
\newcommand{\NN}{{\mathcal N}}
\newcommand{\RR}{{\mathcal R}}
\newcommand{\FF}{{\mathcal F}}
\newcommand{\HHH}{{\mathcal H}}

\newcommand{\diam}{{\textup{diam}}}
\newcommand{\ra}{{\rightarrow}}
\newcommand{\lra}{{\longrightarrow}}
\newcommand{\dist}{{\textup{dist}}}
\newcommand{\Int}{{\textup{Int}}}
\newcommand{\Id}{{\textup{Id}}}
\newcommand{\Cl}{{\textup{Cl}}}
\newcommand{\norm}{\parallel}
\newcommand{\card}{\textup{card}}
\newcommand{\diff}{\textup{Diff}}
\newcommand{\homeo}{\textup{Homeo}}

\newcommand{\filled}{\textup{Fill}}

\newcommand{\SL}{\textup{SL}}
\newcommand{\SO}{\textup{SO}}
\newcommand{\fix}{\textup{Fix}}
\newcommand{\area}{\textup{area}}
\newcommand{\Mod}{\textup{mod}}
\newcommand{\Mob}{\textup{M\"ob}}
\newcommand{\dil}{\textup{dil}}
\newcommand{\gl}{\textup{GL}}
\newcommand{\vol}{\textup{Vol}}
\newcommand{\lin}{\textup{Lin}}

\newcommand{\tu}{\textup}

\title{Minimal Sets of Non-Resonant Torus Homeomorphisms}

\author[F.H. Kwakkel]{Ferry Kwakkel}

\address{F. H. Kwakkel \\ 
University of Warwick \\ 
Mathematics Institute \\ 
CV4 7AL, Coventry \\ 
United Kingdom}

\email{f.h.kwakkel@warwick.ac.uk} 

\begin{abstract}
As was known to H. Poincar\'e, an orientation preserving circle homeomorphism without periodic points is either minimal or has no dense orbits, and every orbit accumulates on the unique minimal set. In the first case the minimal set is the circle, in the latter case a Cantor set. In this paper we study a two-dimensional analogue of this classical result: we classify the minimal sets of non-resonant\index{non-resonant torus homeomorphism} torus homeomorphisms; that is, torus homeomorphisms isotopic to the identity for which the rotation set is a point with rationally independent irrational coordinates. 
\end{abstract}



\maketitle

\section{Introduction and statement of results}

Let $\T^1 = \R \slash \Z$ and $f \colon \T^1 \rightarrow \T^1$ an orientation preserving circle homeomorphism. A lift $F \colon \R \rightarrow \R$ 
of $f$ satisfies $f \circ p_1 = p_1 \circ F$, with $p_1 \colon \R \rightarrow \T^1$ the canonical projection. The number
\begin{equation}
\rho(F,x) := \lim_{n \ra \infty} \frac{F^n(x) - x }{n},
\end{equation}
exists for all $x \in \R$, is independent of $x$ and well defined up to an integer; that is, if $F$ and $\widehat{F}$ are two lifts of $f$ then 
$\rho(F) - \rho(\widehat{F}) \in \Z$. The number $\rho(f):= \rho(F,x)\mod \Z$ is called the {\em rotation number}\index{rotation number} 
of $f$ and $\rho(f) \in \Q$ if and only if $f$ has periodic points. Denote $r_{\theta} \colon \T^1 \ra \T^1$ the rigid rotation of the circle with rotation number $\theta$. The following classical result classifies the possible topological dynamics of orientation preserving homeomorphisms of the circle without periodic points~\cite{poincare_1, poincare_2, poincare_3}. 

\begin{poincare}
Let $f \colon \T^1 \rightarrow \T^1$ be an orientation preserving homeomorphism such that $\rho(f) \in \R \setminus \Q$.
Then 
\begin{enumerate}
\item[(i)] if $f$ is transitive then $f$ is conjugate to the rigid rotation $r_{\rho(f)}$, and 
\item[(ii)] if $f$ is not transitive then $f$ is semi-conjugate to the rotation $r_{\rho(f)}$ via a non-invertible continuous monotone map. 
\end{enumerate}
Moreover, $f$ has a unique minimal set $\MM$, which is the circle $\T^1$ in case \textup{(i)}, or a Cantor set in case \textup{(ii)}
and $\MM = \Omega(f) = \omega(x) = \alpha(x)$ for all $x \in \T^1$.
\end{poincare}

Here, $\Omega(f)$ is used to denote the {\em non-wandering set} of $f$ and $\omega(x), \alpha(x)$ the {\em omega- and alpha-limit set} of $f$ relative to $x \in \T^1$. Every connected component $I$ of the complement of the Cantor minimal set is a wandering interval, i.e. $f^n (I) \cap I = \emptyset$, for all $n \neq 0$. Given a Cantor set in the circle, there exists a circle homeomorphism with any given irrational rotation number that has this Cantor set as its minimal set $\MM$. This fact was first explicitly mentioned by Denjoy~\cite{denjoy}, but essentially known already by Bohl~\cite{bohl} and Kneser~\cite{kneser}. Denjoy~\cite{denjoy} proved that an orientation preserving circle diffeomorphism $f \in \diff^2(\T^1)$ with irrational rotation number is necessarily transitive and hence can not have a wandering interval, see also~\cite{herman79} where 
these ideas are further developed.

\medskip 

The key feature of orientation preserving circle homeomorphisms without periodic points is that it has an irrational rotation number which is independent of the basepoint, where the rotation with the corresponding rotation number is minimal. A natural generalization to dimension two is as follows. Let $\T^2 = \R^2 \slash \Z^2$, where $p \colon \R^2 \ra \T^2$ is the canonical projection mapping. We denote $\homeo(\T^2)$ the class of homeomorphisms of the torus and $\homeo_0(\T^2) \subset \homeo(\T^2)$ the subclass of homeomorphisms isotopic to the identity. Given an element $f \in \homeo_0(\T^2)$, we denote $F \colon \R^2 \ra \R^2$ a lift to the cover. Any two different lifts $F, \widehat{F}$ of $f$ differ by an integer translation, that is, $F(z) = \widehat{F}(z) + (n,m)$ where $(n,m) \in \Z^2$. Given a lift $F$ and $\widetilde{z} \in \R^2$, define
\[ \rho(F) = \bigcap_{m=1}^{\infty} \Cl \left( \bigcup_{n=m}^{\infty} \left\{ \frac{F^n(\widetilde{z}) - \widetilde{z}}{n} \right\} ~|~ \widetilde{z} \in \R^2 \right) \subset \R^2. \]
The {\em rotation set} of $f$ is defined as $\rho(f) = \rho(F) \mod \Z^2$. In words, the rotation set collects all limit points, modulo $\Z^2$, of sequences of the form 
\[ \frac{F^{n_k}(\widetilde{z}) - \widetilde{z}}{n_k}, \]
where $n_k \ra \infty$, for $k \ra \infty$ and $\widetilde{z} \in \R^2$. The rotation set of a homeomorphism $f \in \homeo_0(\T^2)$ is in general no longer a single point, but a convex connected closed set, see~\cite{misi}.
We define 
\begin{equation} 
\homeo_*(\T^2) \subset \homeo_0(\T^2),
\end{equation}
the class of homeomorphisms isotopic to the identity for which the rotation set $\rho(f) = (\alpha, \beta)\mod \Z^2$ is a single point where the numbers $1, \alpha, \beta$ are rationally independent. These homeomorphisms are said to be {\em non-resonant torus homeomorphisms}.

Generalizations of Poincar\'e's Theorem has recently seen much progress through the work of F. B\'eguin, S. Crovisier, T. J\"ager, G. Keller, F. le Roux and J. Stark~\cite{beguin, jager_2, jager_1}, where one considers an analogous class of torus homeomorphisms, namely {\em quasiperiodically forced circle homeomorphisms}, which are torus homeomorphisms of the form $(x, \theta) \mapsto (x+\alpha, g_{\theta}(x)) \mod\Z^2$, with $(x,\theta) \in \T^2$ and $g_{\theta} \colon \T^1 \ra \T^1$ a family of circle homeomorphisms and $\alpha \in \R \setminus \Q$. In~\cite{beguin, jager_2, jager_1}, the appropriate analogues of the results of Poincar\'e and Denjoy in the class of quasiperiodically forced circle homeomorphisms are developed. Analogues of Poincar\'e's Theorem in the setting of conservative torus homeomorphisms are developed by T. J\"ager in~\cite{jager}. In the setting of non-resonant homeomorphisms of the two-torus, introducing differentiability or other geometrical conditions on the homeomorphisms gives rise to analogues of Denjoy's Theorem, see~\cite{no, nv, ns} for results in this direction.

\medskip

To state our results, we need the following definitions. A connected set $X \subset \T^2$ is said to be {\em (un)bounded} according to whether 
a lift $\widetilde{X} \subset \R^2$ is (un)bounded as a subset of $\R^2$, where a lift $\widetilde{X}$ of $X$ is a connected component of 
$p^{-1}(X)$. A compact and connected set is called a {\em continuum}. A continuum $\CC$ in $\T^2$ is called {\em non-separating} if the complement in $\T^2$ is connected. Given a bounded continuum $\CC \subset \T^2$, we define $\filled(\CC) \subset \T^2$, the {\em filled continuum}\index{filled continuum}, the smallest (with respect to inclusion) non-separating bounded continuum containing $\CC$. A bounded non-separating continuum in $\T^2$ is called {\em acyclic}.

\begin{defn}[Extension of a Cantor set]\label{defn_cantor}\index{Cantor set!extension of a Cantor set}
Let $\MM$ be a minimal set for $f \in \homeo_*(\T^2)$, with $\{ \Lambda_i \}_{i \in I}$ be the collection of connected components of $\MM$.
If $\filled(\Lambda_i)$ is acyclic for every $i \in I$, $\filled(\Lambda_i) \cap \filled(\Lambda_j) = \emptyset$ if $i \neq j$ and 
there exists a continuous $\phi: \T^2 \ra \T^2$, homotopic to the identity, and an $\widehat{f} \in \homeo_*(\T^2)$, such that 
\begin{enumerate}
\item[\textup{(i)}] $\phi \circ f = \widehat{f} \circ \phi$, i.e. $f$ is semi-conjugate to $\widehat{f}$, and 
\item[\textup{(ii)}] $\widehat{\MM} := \phi( \widehat{\QQ} ) \subset \T^2$ is a Cantor minimal set for $\widehat{f}$,
\end{enumerate}
where $\widehat{\QQ} = \bigcup_{i \in I} \filled(\Lambda_i)$, then we say $\MM$ is an {\em extension of a Cantor set}.
\end{defn}

Put in words, $\MM$ is an extension of a Cantor set, if the semi-conjugacy $\phi$ between $f$ and $\widehat{f}$ sends the collection of filled in components of $\MM$ to points in a one-to-one fashion, and the corresponding totally disconnected set $\widehat{\MM}$ is a Cantor minimal set of the factor $\widehat{f}$. An extension of a Cantor set is called {\em non-trivial}, if there exist components of $\MM$ that are not singletons. 

Further, we define the following. A {\em disk} $D \subset \T^2$ in the torus is an injection by a homeomorphism of the open unit disk $\D^2 \subset \R^2$ into the torus. An {\em annulus} $A \subset \T^2$ is an injection by a homeomorphism of the open annulus $\mathbb{S}^1 \times (0,1)$ into the torus, where $A$ is said to be {\em essential} if the inclusion $A \hookrightarrow \T^2$ induces an injection of $\pi_1(A)$ into $\pi_1(\T^2)$. Let us now state our main results. Our first result gives a classification of the possible minimal sets of homeomorphisms in our class $\homeo_*(\T^2)$.

\begin{thma}[Classification of minimal sets]\label{thm_main}\index{minimal set!type I}\index{minimal set!type II}\index{minimal set!type III}
Let $f \in \homeo_*(\T^2)$ and let $\MM$ be a minimal set for $f$. Let $\{ \Sigma_k \}_{k \in \Z}$ be the connected components of the complement 
$\MM$ in $\T^2$. If $\MM \neq \T^2$, then either
\begin{enumerate}
\item[\textup{(I)}] $\{ \Sigma_k \}$ is a collection of disks,
\item[\textup{(II)}] $\{ \Sigma_k \}$ is a collection of essential annuli and disks,
\item[\textup{(III)}] $\MM$ is an extension of a Cantor set. 
\end{enumerate}
\end{thma}

We prove this result in section~\ref{sec_class}. It follows from the proof of Theorem A that 

\begin{cor}[Structure of orbits; type I and II]\label{cor_limit_sets}
Let $f \in \homeo_*(\T^2)$ with a minimal set $\MM$ of type I or II. Then 
\begin{equation}\label{eq_thm_orbits}
\MM = \Omega(f) = \omega(z) = \alpha(z),
\end{equation} 
for all $z \in \T^2$. In particular, $\MM$ is unique.
\end{cor}

In~\cite[Thm 1.2]{beguin}, F. B\'eguin, S. Crovisier, T. J\"ager and F. le Roux construct a counterexample to Corollary~\ref{cor_limit_sets}
in the case where $\MM$ is of type III in the setting of quasiperiodically forced circle homeomorphisms. Formulated in our terminology, this 
result reads 

\begin{counterexample}[Structure of orbits; type III~\cite{beguin}]\label{CA_1}
There exist homeomorphisms $f \in \homeo_*(\T^2)$ which have a unique Cantor minimal set $\MM$ (and are thus of type III), but are transitive.
\end{counterexample}

In other words, $\MM \neq \T^2$ is the unique Cantor minimal set, but $\Omega(f) = \T^2$. Uniqueness of minimal sets of type III homeomorphisms 
has not yet been settled, see Question 1 in the final section of this paper. Further, we have that 

\begin{cor}[Connected minimal sets]\label{cor_connected}
Let $f \in \homeo_*(\T^2)$. If $\MM \neq \T^2$, then $\MM$ is connected if and only if $\MM$ is of type I.
\end{cor}

To state the third corollary, we need the following. Recall that a {\em null-sequence} is a sequence of positive real numbers for which for every given $\epsilon>0$ there exist only 
finitely many elements of the sequence that are greater than $\epsilon$.

\begin{defn}[quasi-Sierpi\'nski set]\label{defn_torus_sierpinski}\index{quasi-Sierpi\'nski set}
A {\em quasi-Sierpi\'nski set} is a continuum $S = \T^2 \backslash \bigcup_{k \in \Z} D_k$ with $\{ D_k \}_{k \in \Z}$ a family of disks such that $\bigcup_{k \in \Z} D_k$ is dense in $\T^2$, and
\begin{enumerate}
\item[a)] $D_k$ is the interior of a closed embedded disk, for every $k \in \Z$, 

\item[b)] $\Cl(D_k) \cap \Cl(D_{k'})$ is at most a single point if $k \neq k'$, and

\item[c)] $\diam( D_k )$, $k \in \Z$, is a null-sequence.
\end{enumerate}
If property b) above is replaced by the condition that $\Cl(D_k) \cap \Cl(D_{k'}) = \emptyset$, if $k \neq k'$, then we refer to $S$ as 
a {\em Sierpi\'nski set}.\index{Sierpi\'nski set}
\end{defn}

A closed subset of a topological space is {\em locally connected}, if every of its points has arbitrarily small connected neighbourhoods.
Requiring a minimal set to be locally connected, reduces the list of Theorem A to one type of non-trivial minimal set.

\begin{cor}[Locally connected minimal sets]\label{cor_loc_conn}
Let $f \in \homeo_*(\T^2)$ and suppose that the minimal set $\MM$ of $f$ is locally connected. 
Then either $\MM = \T^2$ or $\MM$ is a quasi-Sierpi\'nski set.
\end{cor}

This result was (essentially) proved by A. Bi\'s, H. Nakayama and P. Walczak in~\cite{bis_1}. We show how this result, for our class of homeomorphisms, can be recovered from Theorem A above, and our line of approach is different. Rather than assuming the minimal set is locally connected from the start as in~\cite{bis_1}, here we show that most of the minimal sets of Theorem A are not locally connected, ultimately arriving at the only possible locally connected minimal set, a quasi-Sierpi\'nski set.

Our second result says that the classification of Theorem A is sharp in the following sense. 

\begin{thmb}[Existence of minimal sets]
Every type of minimal set Theorem A allows is realized by homeomorphisms in $\homeo_*(\T^2)$.
\end{thmb}

This result is proved by a number of examples constructed in section~\ref{sec_existence} below. Let us briefly discuss these. It is well-known there exist homeomorphisms $f \in \homeo_*(\T^2)$ for which the minimal set is a Sierpi\'nski set. The first example given, a locally connected quasi-Sierpi\'nski minimal set that is not a Sierpi\'nski set, is known~\cite{bis_1}. The examples constructed below are a minimal set for which the complement is a single unbounded disk (type I), a minimal set for which the complement components are essential annuli and bounded disks (type II) and examples of rather exotic non-trivial extensions of Cantor sets (type III), where the minimal sets constructed are homeomorphic to Cantor dust interspersed with various continua. 

\section{Classification of minimal sets}\label{sec_class}

In what follows, let $f \in \homeo_*(\T^2)$ with $\MM$ a minimal set of $f$. This section is devoted to the proof of Theorem A.

\subsection{Preliminary results}\label{sec_defns}

We first set notation and recollect several basic results to be used in the remainder of this paper. 

\medskip

Let us first recall the following. A vector $(\alpha, \beta) \in \R^2$ is {\em irrational}, if the numbers $1, \alpha, \beta$ are rationally independent; that is, if the only solution over the integers of 
\begin{equation}\label{eq_numbers_condition}
N_1 + N_2 \alpha + N_3 \beta = 0
\end{equation}
is $N_1 = N_2 = N_3 = 0$. The translation $\tau \colon \T^2 \ra \T^2$ corresponding to $(\alpha, \beta)$, 
\begin{equation}\label{eq_tau}
\tau \colon (x,y) \mapsto (x + \alpha, y + \beta)\mod \Z^2, 
\end{equation}
is minimal if and only if the vector $(\alpha, \beta)$ is irrational. The class of homeomorphisms of the torus $\T^2$ isotopic to the identity with rotation set consisting of a single irrational vector will be denoted by $\homeo_*(\T^2)$. It is easy to see that a homeomorphism $f \in \homeo_*(\T^2)$ has no periodic points. \index{irrational vector}

\begin{lem}\label{lem_fix_1}
Let $f \in \homeo_*(\T^2)$. If $X \subset \T^2$ is a bounded connected set, then $f^n(X) \neq X$, for all $n \neq 0$.
\end{lem}

\begin{proof}
If $X \subset \T^2$ is bounded and $f^N(X) = X$, for some $N \neq 0$, we can take a lift $F$ of $f$ and a lift $\widetilde{X}$ of $X$ such 
that $F^N(\widetilde{X}) = \widetilde{X}$. Let $\widetilde{z} \in \widetilde{X}$. As $\widetilde{X}$ is bounded, 
we must have that $\rho(F, \widetilde{z}) = (0,0)$ and thus $\rho(f,z) = (0,0)\mod \Z^2$, where $z = p(\widetilde{z})$, 
contrary to our assumption on the rotation set.
\end{proof}

In other words, if $X \subset \T^2$ is a connected and $f$-invariant set, then $X$ is necessarily unbounded. In what follows, let 
\begin{equation}\label{eq_int_trans}
T_{p,q} \colon \R^2 \ra \R^2, ~~ T_{p,q}(x,y) = (x+p, y+q),
\end{equation}
where $(p, q) \in \Z^2$.

\begin{lem}\label{lem_fix_2}
Let $f \in \homeo_*(\T^2)$ and $F$ a lift of $f$. Let $D \subset \R^2$ be a closed topological disk. 
Then there exists no $N \neq 0$, and $(p,q) \in \Z^2$, such that 
\begin{equation*}
F^N(D) \subseteq T_{p,q}(D) \quad \textup{or} \quad T_{p,q}(D) \subseteq F^N(D).
\end{equation*}
\end{lem}

\begin{proof}
Suppose that there exists an $N \neq 0$ and $(p,q) \in \Z^2$ such that $F^N(D) \subseteq T_{p,q}(D)$. Choosing a different lift 
$\widehat{F}$ if necessary, we may assume that $\widehat{F}^N(D) \subseteq D$. By the Brouwer Fixed Point Theorem, $\widehat{F}^N$ has a fixed point
on $D$, and thus $f$ has a periodic point, contrary to our assumptions. The case where $T_{p,q}(D) \subseteq F^N(D)$ follows by considering the inverse $F^{-1}$.
\end{proof}

Next, we turn to the topology of domains in the torus. In the subsequent proof, the various topological types of domains on the torus play an important role. In what follows, a {\em domain}\index{domain} is an open connected set. 
Let $\gamma \subset \T^2$ be an essential simple closed curve. We say the curve $\gamma$ has homotopy type $(p,q)$ if $\gamma$ lifts to a curve $\widetilde{\gamma} \in \R^2$ 
such that, up to a suitable translation, $\widetilde{\gamma}$ connects the lattice points $(0,0) \in \Z^2$ and $(p,q) \in \Z^2$ with $p$ and $q$ coprime. Then $\widetilde{\gamma}$ is periodic in the sense that
\begin{equation}\label{eq_trans_curve}
\widetilde{\gamma} = \bigcup_{n \in \Z} T^n_{p,q}(\eta), 
\end{equation}
where $\eta \subset \widetilde{\gamma}$ is the arc connecting $(0,0)$ and $(p,q)$. Let $D \subset \T^2$ be a domain. The inclusion $D \hookrightarrow \T^2$ naturally induces an injection of $\pi_1(D)$ into $\pi_1(\T^2)$. This gives rise to the following.

\begin{defn}[Types of domains]\label{defn_domains_types}
A domain $D \subset \T^2$ is said to be {\em trivial}, {\em essential} or {\em doubly essential} according to whether the inclusion of 
$\pi_1(D)$ into $\pi_1(\T^2)$ is isomorphic to $0, \Z$ or $\Z^2$ respectively. 
\end{defn}

\begin{defn}\label{defn_char}
An essential domain $D \subset \T^2$ has characteristic $(p,q)$ if an essential closed curve $\gamma \subset D$ has homotopy type $(p,q)$.
\end{defn}

Note that definition~\ref{defn_char} is well-defined, in the sense that every other essential simple closed curve in $D$ must have the same homotopy type (as otherwise the domain $D$ would be doubly essential). The following lemma relates the notion of a trivial and essential domain to that of a disk and essential annulus in the torus respectively.

\begin{lem}\label{lem_prel_domains}
A domain $D \subset \T^2$, such that $\widetilde{D}$ is simply connected, is trivial (resp. essential) if and only if it is a 
disk (resp. essential annulus) in the torus.
\end{lem}

\begin{proof}
As the {\em if} part is evident, we need only prove the {\em only if} part. First suppose that $D$ is trivial and let $\widetilde{D}$ a lift of $D$.
By the Riemann mapping theorem, there exists a biholomorphism $\phi \colon \D^2 \ra \widetilde{D}$,  where $ \D^2 \subset \R^2$ is the unit (Poincar\'e) disk. As $D$ is trivial, no two points in $\widetilde{D}$ are identified under the action of the translation group 
$\Z^2 \subset \R^2$ and thus $p|_{\widetilde{D}}$ is injective, where $p \colon \R^2 \ra \T^2$ is the projection mapping. Therefore, we have that $p \circ \phi \colon \D^2 \ra \T^2$ with $p \circ \phi(\D^2) = D$, and thus $D \subset \T^2$ is a disk.

Next, suppose that $D$ is essential. Then there exists a unique pair $(p,q) \in \Z^2$, with $\gcd(p,q) =1$, such that the translation 
$T_{p,q}$ leaves $\widetilde{D}$ invariant, i.e. $T_{p,q}(\widetilde{D}) = \widetilde{D}$. Further, as $\widetilde{D}$ is simply connected, 
again by the Riemann mapping theorem, there exists a biholomorphism $\phi \colon \D^2 \ra \widetilde{D}$. 
As $T_{p,q} \colon \widetilde{D} \ra \widetilde{D}$ is a biholomorphism, the map 
\[ \mu \colon \D^2 \longrightarrow \D^2, ~\mu = \phi^{-1} \circ T_{p,q} \circ \phi \] 
is itself a biholomorphism and thus a M\"obius transformation. Moreover, as $T_{p,q}$ does not fix any point in $\widetilde{D}$, $\mu$ does not fix any point in $\D^2$ and thus $\mu$ is either a hyperbolic or a parabolic M\"obius transformation. 

As is well known, $\D^2 \slash \langle \mu \rangle$ is conformally equivalent to an annulus if $\mu$ is hyperbolic, and conformally equivalent to the once punctured disk $\D^2 \setminus \{0\}$ if $\mu$ is parabolic (see e.g.~\cite{abikoff}). 
Thus both are topologically equivalent to an annulus $\mathbb{S}^1 \times (0,1)$, therefore so is $\widetilde{D} \slash \langle T_{p,q} \rangle$. As $\widetilde{D}$ admits no translations other than (multiples of) $T_{p,q}$ that leave $\widetilde{D}$ invariant, 
the continuous projection $p$ restricted to  $\widetilde{D} \slash \langle T_{p,q} \rangle$ into the torus $\T^2$ is an injection and thus $D$ is indeed topologically equivalent to the annulus $\mathbb{S}^1 \times (0,1)$.
\end{proof}

We recall some standard results from decomposition theory, to be used in the proof of Theorem A. In the following statements, let $M$ be a closed surface.

\begin{defn}\label{defn_upper_semi}
A collection $\UU = \{ \UU_i\}_{i \in I}$ of continua in a surface $M$ is said to
be {\em upper semi-continuous} if the following holds:
\begin{enumerate}
\item[(1)] If $\UU_i, \UU_j \in \UU$, then $\UU_i \cap \UU_j = \emptyset$.
\item[(2)] If $\UU_i \in \UU$, then $\UU_i$ is non-separating. 
\item[(3)] We have that $M = \bigcup_{i \in I} \UU_i$.
\item[(4)] If $\UU_{i_k}$ with $k \in \N$ is a sequence that has the Hausdorff limit $\CC$, then there
exists $\UU_j \in \UU$ such that $\CC \subset \UU_j$.
\end{enumerate}
\end{defn}

In a compact metric space, every Hausdorff limit of continua is again a continuum. We have the following classical result, see for example~\cite{whyburn} or~\cite{daver}.

\begin{moore}
Let $\UU$ be an upper semi-continuous decomposition of $M$ so that every element of $\UU$ is acyclic. Then there is a continuous map
$\phi: M \ra M$ that is homotopic to the identity and such that for every $z \in M$, we have that $\phi^{-1}(z) = \UU_i$ for some element 
$\UU_i \in \UU$.
\end{moore}

The following follows from Moore's Theorem, see also~\cite{bruin}. We give a proof here for the convenience of the reader.

\begin{lem}\label{lem_moore_semic_conj}
Given an upper semicontinuous decomposition $\UU$ of $\T^2$ into acyclic elements and a $f \in \homeo_0(\T^2)$, with the property that $f$ sends elements of $\UU$ onto elements of $\UU$. 
Then the map $\widehat{f} \colon \T^2 \ra \T^2$ defined by $\phi \circ f(z) = \widehat{f} \circ \phi(z)$, for every $z \in \T^2$, is an element of $\homeo_0(\T^2)$. In other words, $f$ is semi-conjugate to $\widehat{f}$ through $\phi$.
\end{lem}

\begin{proof}
By Moore's Theorem, the map $\phi \colon \T^2 \ra \T^2$ sending elements of $\UU$ to points is continuous. The map $\widehat{f}$, where $\phi \circ f(z) = \widehat{f} \circ \phi(z)$ is one-to-one as $f$ sends elements of $\UU$ onto 
elements of $\UU$. To prove that $\widehat{f}$ is continuous, we observe that, whenever $\CC \subset \T^2$ is closed, then so are $\CC'=(\phi \circ f)^{-1}(\CC)$ and $\CC'' = \phi (\CC')$, since $\phi$ and $f$ are continuous and 
$\T^2$ is compact. Since $\CC''=\widehat{f}^{-1}(\CC)$, this shows that $\widehat{f}$ is continuous and hence a homeomorphism, again by compactness of $\T^2$.
\end{proof}

\subsection{Topology of the domains $\Sigma_k$}\label{subsec_top_Sigma}

In what follows, let $\Sigma = \Sigma_k$ be any element of $\{ \Sigma_k\}$, the collection of connected components of the complement of $\MM$,
a minimal set of an element $f \in \homeo_*(\T^2)$. Further, let $\widetilde{d}(\cdot, \cdot)$ be the standard Euclidean metric on $\R^2$ 
and let $d(\cdot, \cdot)$ the (induced) metric on $\T^2$.

\begin{lem}\label{lem_wander}
If $f^n(\Sigma) \cap \Sigma = \emptyset$ for all $n \neq 0$, then $\Sigma$ is a disk or an essential annulus.
\end{lem}

\begin{proof}
First, suppose that $\Sigma$ is doubly essential and let $\gamma, \gamma' \subset \Sigma$ be two non-homotopic essential simple closed curves. 
As $f$ is homotopic to the identity, the homotopy classes of $f(\gamma)$ and $\gamma$ are equal. As every two non-homotopic simple closed curves on the torus intersect, we have that $f(\gamma) \cap \gamma' \neq \emptyset$. Therefore $f(\Sigma) \cap \Sigma \neq \emptyset$ and thus $\Sigma$ has to be either trivial or essential. 

Thus let $\Sigma$ be a trivial or essential domain and let $\widetilde{\Sigma}$ be a lift of $\Sigma$. In order to show that $\Sigma$ is a disk or essential annulus respectively, by Lemma~\ref{lem_prel_domains}, it suffices to show that $\widetilde{\Sigma}$ is simply connected.
To prove this, suppose to the contrary that $\widetilde{\Sigma}$ is not simply connected. Then there exists a simple closed curve 
$\gamma \subset \widetilde{\Sigma}$ such that the open disk $D_{\gamma}$ with boundary curve $\gamma$ has the property that 
\begin{equation}
D_{\gamma} \cap p^{-1}(\MM) \neq \emptyset.
\end{equation}
Let $F$ be a lift of $f$. As every point in $\MM$ is recurrent, there exists a subsequence $n_k$ such that $f^{n_k}(z) \ra z$ for 
$k \ra \infty$. Therefore, by passing to a subsequence if necessary, we may assume that for all $k \geq 1$, we have that 
\begin{equation}\label{eq_top_sigma_1}
F^{n_k}(D_{\gamma}) \cap T_{p_k,q_k}(D_{\gamma}) \neq \emptyset,
\end{equation}
for certain $(p_k,q_k) \in \Z^2$. Given~\eqref{eq_top_sigma_1}, there are two possibilities. For a given $k \geq 1$, we have that either 
\begin{enumerate}
\item[\textup{(a)}] $F^{n_k} ( D_{\gamma} ) \subset T_{p_k,q_k}(D_{\gamma})$ or $T_{p_k,q_k}(D_{\gamma}) \subset F^{n_k} ( D_{\gamma} )$, or
\item[\textup{(b)}] $F^{n_k} ( \gamma ) \cap T_{p_k,q_k}(\gamma) \neq \emptyset$.
\end{enumerate}
Case (a) can be excluded as, by Lemma~\ref{lem_fix_2}, this yields periodic points for $f$. 
Furthermore, case (b) is ruled out as this implies that 
\begin{equation}
F^{n_k} ( \widetilde{\Sigma}) \cap T_{p_k,q_k}( \widetilde{\Sigma} ) \neq \emptyset,
\end{equation}
implying that $f^{n_k}(\Sigma) \cap \Sigma \neq \emptyset$, contrary to our assumption. Therefore, $\widetilde{\Sigma}$ must be simply 
connected indeed.
\end{proof}

In what follows, a fundamental domain of $\T^2$ is defined as the standard square $[0,1] \times [0,1] \subset \R^2$ and the integer translates 
thereof.

\begin{lem}\label{lemma_disk}
If $\Sigma$ is trivial, then $\Sigma$ is a disk. Moreover, if $\Sigma$ is bounded, then $f^n(\Sigma) \cap \Sigma = \emptyset$ for all $n \neq 0$.
\end{lem}

\begin{proof}
Because $\MM$ is invariant, we have that either: (a) $f^n(\Sigma) \cap \Sigma = \emptyset$ for all $n \neq 0$ or (b) $f^N(\Sigma) = \Sigma$ 
for some $N \neq 0$. In case (a), $\Sigma$ is a disk by Lemma~\ref{lem_wander}. In case (b), it follows from Lemma~\ref{lem_fix_1} that $\Sigma$
is necessarily unbounded. 

Thus we need to show that for unbounded $\Sigma$, we have that $\widetilde{\Sigma}$ is simply connected, if $f^N(\Sigma) = \Sigma$ for 
some $N \neq 0$. We may as well assume that $N=1$. Take a lift $F$ of $f$ such that $F(\widetilde{\Sigma}) = \widetilde{\Sigma}$. If $\widetilde{\Sigma}$ is not simply connected, then there exists a simple closed curve $\gamma \subset \widetilde{\Sigma}$ such that the disk $D_{\gamma} \subset \R^2$ with boundary curve $\gamma$ has the property that $D_{\gamma} \cap p^{-1}(\MM) \neq \emptyset$. Similarly to 
Lemma~\ref{lem_wander}, there exists a subsequence $n_k$ such that, for $k \geq 1$, we have that 
\begin{equation}\label{eq_top_disk_1}
F^{n_k}(D_{\gamma}) \cap T_{p_k,q_k}(D_{\gamma}) \neq \emptyset,
\end{equation}
for certain $(p_k,q_k) \in \Z^2$. As 
\[ \rho(f) = \rho(f,z) = (\alpha, \beta)\mod \Z^2 \neq (0,0)\mod \Z^2, \] 
for every $z \in \T^2$, it follows that $\widetilde{d}(F^{n_k}(\widetilde{z}), \widetilde{z}) \ra \infty$, for $k \ra \infty$. In particular, passing to a subsequence once again, we may assume that $F^{n_k}(\widetilde{z})$ is contained in a fundamental domain different from that of $\widetilde{z}$, for all $k \geq 1$. Condition~\eqref{eq_top_disk_1} gives again the two possiblities (a) and (b) of Lemma~\ref{lem_wander} and we can exclude case (a) as this would yield periodic points for $f$. Therefore, for all $k \geq 1$,~\eqref{eq_top_disk_1} reduces to the condition that 
\begin{equation}
F^{n_k} ( \gamma ) \cap T_{p_k,q_k}(\gamma) \neq \emptyset,
\end{equation}
for some $(p_k,q_k) \in \Z^2$. Thus for every $k \geq 1$, we have that
\begin{itemize}
\item[\tu{(i)}] $F^{n_k} ( \gamma ) \cap T_{p_k,q_k}(\gamma) \neq \emptyset$, 
\item[\tu{(ii)}] $F^{n_k} ( \gamma )$ lies in a fundamental domain different from that of $\gamma$, and 
\item[\tu{(iii)}] $F^{n_k} ( \gamma ) \subset \widetilde{\Sigma}$. 
\end{itemize}
Condition (iii) follows simply from the fact that $\widetilde{\Sigma}$ is $F$-invariant and $\gamma \subset \widetilde{\Sigma}$. 
Fix any $k \geq 1$ and choose $\widetilde{w} \in F^{n_k} ( \gamma ) \cap T_{p_k,q_k}(\gamma)$ and let 
$\widetilde{w}' = T^{-1}_{p_k,q_k}(\widetilde{w}) \in \gamma$. As $\widetilde{\Sigma}$ is a domain, it is path-connected and thus there exists an 
arc $\eta \subset \widetilde{\Sigma}$ connecting $\widetilde{w}$ and $\widetilde{w}'$. As these endpoints lie in different fundamental domains of $\T^2$, $\eta$ projects under $p$ to an essential closed curve, as its endpoints differ by an integer translate. 
However, this contradicts our assumption that $\Sigma$ is trivial (and thus does not contain any essential simple closed curves). This contradiction shows that $\Sigma$ must be simply connected and this completes the proof. 
\end{proof}

Using the irrationality of the rotation vector $(\alpha, \beta)$, we now deduce the following.

\begin{lem}\label{lemma_cylinder}
If $\Sigma$ is essential, then $\Sigma$ is an essential annulus and $f^n(\Sigma) \cap \Sigma = \emptyset$ for all $n \neq 0$.
\end{lem}

\begin{proof}
It suffices to show that, if $\Sigma$ is essential, then $f^n(\Sigma) \cap \Sigma \neq \emptyset$ for all $n \neq 0$. It then follows 
from Lemma~\ref{lem_wander} that $\Sigma$ is an essential annulus. Assume that $\Sigma$ has characteristic $(p,q)$. We will show that, by our choice of translation number, $f^N$ can not fix an essential domain, for all $N \neq 0$. To derive a contradiction, suppose there exist an $N \neq 0$ such that $f^N(\Sigma) = \Sigma$. Without loss of generality, we may assume that $N=1$, i.e. that $f(\Sigma) = \Sigma$. Let $\gamma \subset \Sigma$ be an essential simple closed curve and let $\widetilde{\gamma}$ be a lift of $\gamma$. We may assume that $\widetilde{\gamma}$ intersects $(0,0) \in \R^2$, and by definition it also intersects $(p,q) \in \R^2$, where $p,q \in \Z$ and $\gcd(p,q)=1$. The arc $\eta \subset \widetilde{\gamma}$ connecting $(0,0)$ and $(p,q)$ is compact and therefore bounded. Therefore, the curve $\widetilde{\gamma}$ divides $\R^2$ into two unbounded connected components $\HHH_l$ and $\HHH_r$, homeomorphic to half-planes, so that $\R^2 \backslash \widetilde{\gamma} = \HHH_l \cup \HHH_r$ and $\HHH_l \cap \HHH_r = \emptyset$. Further, as $\gamma$ is a simple closed curve, any integer translate $\widetilde{\gamma}' = T_{p',q'}(\widetilde{\gamma})$, where $(p',q')$ is not an integer multiple of $(p,q)$, has the property that $\widetilde{\gamma}' \cap \widetilde{\gamma} = \emptyset$. This follows from the fact that $\Sigma$ is essential, but not doubly essential; if $\widetilde{\gamma}' \neq \widetilde{\gamma}$, then there exists an arc $\zeta \subset \widetilde{\Sigma}$ connecting $(0,0)$ to a point $(p',q') = T_{p',q'}(0,0)$, which is not a multiple of $(p,q)$, the projection of $\zeta$ under $p$ would lie in a homotopy class other than that of $\gamma$, implying that $\Sigma$ would be doubly essential, contrary to our assumption. Therefore, we can choose integer translates $\widetilde{\gamma}_l$ and $\widetilde{\gamma}_r$ of $\widetilde{\gamma}$ contained in $\HHH_l$ and $\HHH_r$ respectively and we can define $\Gamma \subset \R^2$ to be the infinite strip bounded by $\widetilde{\gamma}_l \cup \widetilde{\gamma}_r$. 

We claim that $\widetilde{\Sigma} \subset \Gamma$. Indeed, if $\widetilde{\Sigma} \cap \Gamma^c \neq \emptyset$, where $\Gamma^c := \R^2 \setminus \Gamma$, 
then $\widetilde{\Sigma} \cap \left( \widetilde{\gamma}_l \cup \widetilde{\gamma}_r \right) \neq \emptyset$. Suppose that $\widetilde{\Sigma} \cap \widetilde{\gamma}_l \neq \emptyset$. The case where $\widetilde{\Sigma} \cap \widetilde{\gamma}_r \neq \emptyset$ (or both) is similar. 
Let $\widetilde{z}' \in \widetilde{\Sigma} \cap \widetilde{\gamma}_l$. Because $\widetilde{\gamma} \subset \widetilde{\Sigma}$ and $\widetilde{\Sigma}$ is path-connected, there exists an arc $\zeta \subset \widetilde{\Sigma}$ connecting $\widetilde{z}'$ to a point 
$\widetilde{z} \in \widetilde{\gamma}$ such that $z = p(\widetilde{z}') = p(\widetilde{z})$. As the arc $\zeta$ connects two lattice points whose projection lies in a homotopy class different from $\gamma$, this would imply that $\Sigma$ is doubly essential, contrary to our assumption.
Thus $\widetilde{\Sigma} \subset \Gamma$.

To finish the proof, choose a lift $F$ of $f$ such that $F(\widetilde{\Sigma}) = \widetilde{\Sigma}$. As $\Gamma$ is invariant under the translation $T_{p,q}$, and $\widetilde{\Sigma} \subset \Gamma$, we thus must have that 
\begin{equation}
\rho(F,\widetilde{z}) = \lim_{n \ra \infty} \frac{F^n (\widetilde{z}) - \widetilde{z}}{n} = (a,b),~\textup{where}~ \frac{b}{a} = \frac{q}{p},
\end{equation}
for every $\widetilde{z} \in \widetilde{\Sigma}$. As $(a,b) = (\alpha + s, \beta + t)$ for certain $s,t \in \Z$ and $\alpha, \beta \notin \Q$, 
we have that $a = \alpha + s \neq 0$ and $b = \beta + t \neq 0$, and we obtain
\begin{equation}\label{eq_algebra}
\frac{\alpha + s}{\beta + t} = \frac{a}{b} = \frac{p}{q}.
\end{equation}
Rewriting~\eqref{eq_algebra} gives that 
\[ q \alpha - p \beta - (qs - pt) = 0. \] 
As $p,q, qs - pt \in \Z$, with $(p,q) \neq (0,0)$, this gives a non-trivial solution of~\eqref{eq_numbers_condition}, which contradicts the 
irrationality of $(\alpha, \beta)$. 
\end{proof}

We recall that a Cantor set can be characterized topologically as being compact, perfect and totally-disconnected.

\begin{lem}\label{lem_dessential}
If $\Sigma$ is doubly essential, then $\MM$ is an extension of a Cantor set.
\end{lem}

\begin{proof}
Denote $\{\Lambda_i\}_{i \in I}$ the collection of connected components of $\MM$. Because $\Sigma$ is doubly essential, 
$f(\Sigma) \cap \Sigma \neq \emptyset$, hence $f(\Sigma) = \Sigma$. As $f(\Sigma) = \Sigma$, we have that $f(\partial \Sigma) = \partial \Sigma$. Further, as $\partial \Sigma \subset \MM$ and $\partial \Sigma$ is closed, $\partial \Sigma = \MM$ by minimality of $\MM$. Let $\Lambda:=\Lambda_i$ be any connected component of $\MM$, for some $i \in I$. As $\Lambda$ is closed in $\MM$ and $\MM$ is closed in $\T^2$, $\Lambda$ is closed, and thus compact, in $\T^2$. Therefore, $\Lambda$ is a continuum. Further, as $\MM$ is nowhere dense (as $\MM \neq \T^2$), it follows that $\Lambda$ is nowhere dense. 

We need to show that a lift $\widetilde{\Lambda}$ of $\Lambda$ is bounded. As $\Sigma$ is doubly essential, there exist two non-homotopic essential simple closed curves $\gamma, \gamma' \subset \Sigma$. As these curves are non-homotopic, the respective lifts 
$\widetilde{\gamma}, \widetilde{\gamma}' \subset \widetilde{\Sigma}$ of $\gamma$ and $\gamma'$ and the integer translates of these curves tile $\R^2$ into bounded disks. As $\widetilde{\Lambda} \cap \widetilde{\Sigma} = \emptyset$, it follows that $\widetilde{\Lambda}$ has to be contained in one of these bounded disks, implying $\widetilde{\Lambda}$ itself is bounded. 

As $\Lambda$ is a connected component of $\MM$, we must either have that $f^n(\Lambda) \cap \Lambda = \emptyset$ for all $n \neq 0$, 
or $f^N(\Lambda) = \Lambda$ for some finite $N \neq 0$. However, the latter is excluded by Lemma~\ref{lem_fix_1} as $\Lambda$ is bounded 
and thus $f^n(\Lambda) \cap \Lambda = \emptyset$ for all $n \neq 0$. As $\Lambda$ is a bounded continuum, the connected components of 
$\T^2 \setminus \Lambda$ consists of a unique unbounded component and every other component is a bounded disk. Let $D$ be one such disk. 
Then $\Sigma$ is contained in this unbounded component; indeed, if this would not be the case, then we can take a point $z \in D \cap \Sigma$ 
and an essential simple closed curve $\gamma \subset \Sigma$ passing through $z$. As $\MM \cap \Sigma = \emptyset$, and $z \in D$, 
this implies that $\gamma \subset D$, contradicting that $D$ is a disk. We thus conclude that $D \cap \Sigma = \emptyset$.
Further, if $D \cap \MM \neq \emptyset$, then this implies that $\Sigma \cap D \neq \emptyset$ as $\partial \Sigma = \MM$, 
which contradicts our earlier conclusion. In other words, to a component $\Lambda$, we can uniquely adjoin the open disks which, apart 
from the unique doubly essential component containing $\Sigma$, form the connected components of $\T^2 \setminus \Lambda$. This proves that $\filled(\Lambda_i) \cap \filled(\Lambda_j) = \emptyset$ if $i \neq j$, with $\filled(\Lambda_i)$ a bounded non-separating continuum, for 
every $i \in I$.

Let again $\Lambda_i$ be any component of $\MM$ and define $\widehat{\QQ} = \bigcup_{i \in I} \filled(\Lambda_i)$. Define the decomposition $\UU$ 
of $\T^2$ into the continua $\{ \filled(\Lambda_i) \}_{i \in I}$ and singletons in the complement of these continua. In order to show that 
$\MM$ is an extension of a Cantor set, we first show that the decomposition $\UU$ is upper semi-continuous. By Moore's Theorem, this implies there exists a continuous $\phi \colon \T^2 \ra \T^2$ such that $\phi^{-1}(z) = \UU_j$ for every $z \in \T^2$. 
We have already shown that the decomposition $\UU$ satisfies conditions (1), (2) and (3) of definition~\ref{defn_upper_semi}. To prove it satisfies condition (4), we need to show that if a sequence of continua $\UU_{j_k}$, with $k \in \Z$, has Hausdorff limit $\CC$, then $\CC \subset \UU_{j}$ for some $j \in J$. As the statement is obvious if $\CC$ is a singleton, assume $\CC$ to be a non-trivial continuum. Note that every non-degenerate element $\UU_{j_k} \in \UU$ has the property that $\partial \UU_{j_k} \subset \partial \Sigma = \MM$. Further, without loss of generality, we may assume that no element $\partial \UU_{j_k}$ is a singleton and that the elements are mutually disjoint. We first claim that the interior of $\CC$ has to be empty. Indeed, if not, there would exist a subsequence of elements for which the largest open disk contained in the interior of $\UU_{j_k}$ would be bounded from below, contradicting that the torus is compact and the elements mutually disjoint. Therefore, as $\partial \UU_{j_k} \subset \MM$, every point of the Hausdorff limit $\CC$ is the limit point of a sequence of points of $\MM$. As $\MM$ is closed, this implies $\CC$ is itself contained in $\MM$. In particular, as $\CC$ is connected, $\CC$ is contained in a connected component of $\MM$, i.e. $\CC \subset \Lambda_i \subset \UU_j$ for some $j \in J$. So $\UU$ is upper semi-continuous indeed. We have 
already shown that all non-trivial elements of $\UU$ are non-separating and bounded, and thus acyclic. 

Thus, by Moore's Theorem, there exists a continuous $\phi \colon \T^2 \ra \T^2$, homotopic to the identity, such that for every $z \in \T^2$, $\phi^{-1}(z)$ is a unique element of $\UU$. By Lemma~\ref{lem_moore_semic_conj}, as $\UU$ is upper semi-continuous and $f$ sends elements 
of $\UU$ into elements of $\UU$, the mapping $\widehat{f}$ defined by $\phi \circ f = \widehat{f} \circ \phi$ is an element of $\homeo_0(\T^2)$. By a standard argument, $\rho(f) = \rho(\widehat{f}) \mod \Z^2$, thus $\widehat{f} \in \homeo_*(\T^2)$. 
This proves condition (i) of definition~\ref{defn_cantor}.

To prove condition (ii) of definition~\ref{defn_cantor}, we need to show that $\widehat{\MM} := \phi(\widehat{\QQ}) \subset \T^2$ is a
 Cantor minimal set for $\widehat{f}$. As $\MM$ is a minimal set for $f$, $\widehat{\MM}$ is a minimal set for $\widehat{f}$. Further, as $\widehat{\MM}$ is totally disconnected by construction, it suffices to show that $\widehat{\MM}$ is compact and perfect. First, $\widehat{\QQ}$ is compact as the complement $\Sigma$ is open. Because $\phi$ is continuous, $\widehat{\MM}$ is compact. To show $\widehat{\MM}$ is perfect, we observe that, because $\filled(\Lambda_i) \cap \filled(\Lambda_j) = \emptyset$ if $i \neq j$, no element $\filled(\Lambda_i)$ is isolated, as this would imply that a component $\Lambda_i$ is isolated. Therefore, by continuity of $\phi$, no point of $\widehat{\MM}$ is isolated, and thus $\widehat{\MM}$ is perfect. 
\end{proof}

\subsection{Proof of Theorem A}

\begin{proof}[Proof of Theorem A]
To show that the minimal set $\MM$ of $f$ is either of type I, II or III as given above, assume that $\MM \neq \T^2$ and let $\{\Sigma_k\}$ be the collection of connected components of the complement of $\MM$. If no element of $\{\Sigma_k \}$ is doubly essential, then $\Sigma_k$ is either trivial or essential, for all $k \in \Z$. By Lemma~\ref{lemma_disk} and~\ref{lemma_cylinder}, $\{ \Sigma_k \}$ are all disks and/or essential annuli. In case no element $\Sigma_k$ is essential, we have a type I minimal set. In case at least one, and therefore infinitely many, connected components are 
essential, we have a type II minimal set. If for some $k$, $\Sigma_k$ is doubly essential, then $\MM$ is an extension of a Cantor set by Lemma~\ref{lem_dessential} and these correspond to a type III minimal sets. This concludes the proof.
\end{proof}

We finish this section with the proofs of the corollaries stated above.

\begin{proof}[Proof of Corollary~\ref{cor_limit_sets}]
Let $\MM$ be a minimal set of $f$ of type I. It suffices to show that $\MM = \Omega(f)$. Indeed, if this is shown, 
then by minimality of $\MM$ and the inclusions $\alpha(z), \omega(z) \subseteq \Omega(f)$, with $\omega(z), \alpha(z)$ closed and $f$-invariant sets for every $z \in \T^2$, we obtain~\eqref{eq_thm_orbits}. Uniqueness then also follows, as any other minimal set $\MM'$ of $f$ has to be contained in the complement of $\Omega(f)$, which is clearly impossible; if $z \in \MM'$ then $z$ is both recurrent and wandering, which are incompatible conditions to hold simultaneously. 

First, suppose that $\MM$ is of type I. Fix a component $\Sigma := \Sigma_k$. Then $\Sigma$ is a disk. If $\Sigma$ is bounded, then by 
Lemma~\ref{lemma_disk} we have that $f^n(\Sigma) \cap \Sigma = \emptyset$ for all $n \neq 0$, and thus $\Sigma \cap \Omega(f) = \emptyset$.
So it remains to prove the case where $\Sigma$ is an unbounded disk. We may assume that there exists an $N \neq 0$ such that 
$f^N(\Sigma) = \Sigma$, as otherwise we are done by the previous argument. A straightforward modification of the argument of Lemma~\ref{lemma_disk} 
shows that $\Omega(f) \cap \Sigma = \emptyset$ in this case as well. Finally, if $\MM$ is of type II, then all elements of $\{ \Sigma_k \}$ are essential annuli or disks. It suffices to show that if $\Sigma_k$ is an essential annulus, then $\Omega(f) \cap \Sigma_k = \emptyset$. 
This follows from Lemma~\ref{lemma_cylinder} stating that $f^n(\Sigma_k) \cap \Sigma_k = \emptyset$ for all $n \neq 0$, and this finishes the proof.
\end{proof}

\begin{proof}[Proof of Corollary~\ref{cor_connected}]
First, it is readily verified that no minimal set of type II or III is connected. Conversely, let $\MM$ be of type I, so that $\Sigma_k$ is an open topological disk for every $k \in \Z$. In each disk $\Sigma_k$, we can find
a sequence $D^t_k$ of nested disks, i.e. $D_k^t \subset D^{t+1}_k$, embedded in $\Sigma_k$, such that $\Cl(D_k^t)$ is a closed disk and such that $\bigcup_{t \geq 1} D^t_k = \Sigma_k$. We can accomplish this by uniformizing each disk $\Sigma_k$ to the unit disk $\D^2$, taking nested such disks centered at the origin in $\D^2$, and pulling these back to $\Sigma_k$. Define $\Gamma_t = \T^2 \setminus \bigcup_{k \in \Z} D_k^t$. 
We claim that $\Gamma_t$ is connected. Indeed, define the compact sets $\Gamma_t^s = \T^2 \setminus \bigcup_{k = -s}^s D_k^t$. 
Clearly, $\Gamma_t^s$, as the torus with finitely many disjoint disks whose closures are disjoint deleted, is connected. 
As $\Gamma_t^{s+1} \subset \Gamma_t^s$, we have that $\Gamma_t = \bigcap_{s \geq 1} \Gamma_t^s$ is connected as well. 
By the same token, as $\Gamma_{t+1} \subset \Gamma_t$ with $\Gamma_t$ compact and connected for every $t \geq 1$, 
we have that $\MM = \bigcap_{t \geq 1} \Gamma_t$ is connected.
\end{proof}

\begin{proof}[Sketch of the proof of Corollary~\ref{cor_loc_conn}]
First, it is not difficult to show that a locally connected minimal set $\MM$ has to be of type I and that the collection $\{\Sigma_k\}$ can not contain an unbounded disk. By the same arguments as in~\cite[Lemma 7]{bis_1}, one can now show that $\diam(\Sigma_k)$ is a null-sequence. Further, as no minimal set of a homeomorphism of a compact metric space can have cut points by~\cite[Lemma 2]{bis_1}, it follows from e.g.~\cite[Thm 61-4]{kura} that $\Sigma_k$ is the interior of a closed embedded disk.
To show that $\Cl(\Sigma_k) \cap \Cl(\Sigma_{k'})$ consists of at most one point if $k \neq k'$, we reason as follows. Denote $\Sigma = \Sigma_k$ and $\Sigma' = \Sigma_{k'}$ and let $\gamma = \partial \Sigma$ and $\gamma' = \partial \Sigma'$, both simple closed (trivial) curves. Assume, to the contrary, that $\gamma \cap \gamma'$ contains at least two points $z_1, z_2$. There exists an arc $\eta \subset \Cl(\Sigma)$ starting at $z_1$ and ending at $z_2$ such that $\eta \cap \partial \Sigma = \{z_1,z_2\}$. Similarly, there exists an arc $\eta' \subset \Cl(\Sigma')$ starting at $z_1$ and ending at $z_2$ such that $\eta' \cap \partial \Sigma' = \{z_1,z_2\}$. Then $\eta \cup \eta'$ forms a simple closed curve that bounds a disk $D$. As the diameters of $\Sigma$ and $\Sigma'$ tend to zero, $\diam(f^n (D)) \ra 0$ for $|n| \ra \infty$. Furthermore, as $z_1 \neq z_2$, $D$ contains arcs contained in $\partial \Sigma$ and $\partial \Sigma'$ joining $z_1$ and $z_2$ in its interior, we have that $D \cap \MM \neq \emptyset$. As $\diam(f^n(D))$ is a null sequence, for sufficiently large $N$, we have that $f^N(D) \subset D$, which by Lemma~\ref{lem_fix_2} implies $f$ has periodic points. Therefore, $\Cl(\Sigma) \cap \Cl(\Sigma')$ can consist of at most a single point and thus $\MM$ is indeed a quasi-Sierpi\'nski set. 
\end{proof}

\section{Existence of minimal sets}\label{sec_existence}

Having given a classification of the possible minimal sets of homeomorphisms $f \in \homeo_*(\T^2)$, this section is aimed at constructing such homeomorphisms admitting a minimal set of every type Theorem A allows. 
More precisely, we show there exist $f \in \homeo_*(\T^2)$ such that its minimal set $\MM$

\begin{enumerate}
\item is a quasi-Sierpi\'nski set, but not a Sierpi\'nski set,
\item is such that the complement consists of a single unbounded disk,
\item is such that the complement consists of essential annuli and disks,
\item is a non-trivial extension of a Cantor set.
\end{enumerate}

It is well-known there exist homeomorphisms $f \in \homeo_*(\T^2)$ for which the minimal set is a Sierpi\'nski set. Example~\ref{ex_quasi_sierp}, 
constructed in~\cite{bis_1}, is derived from the Sierpi\'nski set, see also section~\ref{ex_type_I} below for a discussion of this example. 
The remainder of this section is devoted to the construction of Example~\ref{ex_unbdd_disk} (type I), Example~\ref{ex_cylinders} (type II) and 
Example~\ref{ex_type_III} (type III). Combined these examples prove Theorem B.

\subsection{Homeomorphisms semi-conjugate to an irrational translation}

There is a natural subclass of $\homeo_*(\T^2)$, namely those homeomorphisms that are semi-conjugate to an irrational translation of the torus. 
Indeed, a standard argument shows that an element $f \in \homeo_0(\T^2)$ semi-conjugate to a translation $\tau$, through a continuous map 
homotopic to the identity, has the property that $\rho(f) = \rho(\tau)\mod \Z^2$. Given a continuous map $\pi \colon \T^2 \ra \T^2$, we call 
the set of points 
\begin{equation}
\RR_{\pi} = \left \{ z \in \T^2 ~|~ \# \left( \pi^{-1}( \pi(z)) \right) = 1 \right \} \subset \T^2,
\end{equation}
the {\em regular set} of $\pi$. \index{regular set}

\begin{defn}\label{defn_semi_conj}
We define the class $\homeo_{\#}(\T^2) \subset \homeo_*(\T^2)$ the class of homeomorphisms which satisfy the following:
\begin{enumerate}
\item[\textup{(i)}] $f$ is isotopic to the identity, i.e. $f \in \homeo_0(\T^2)$, 
\item[\textup{(ii)}] there exists a monotone\footnote{A map is said to be monotone if every point-inverse is connected.} and 
continuous $\pi:\T^2 \ra \T^2$, homotopic to the identity, and an irrational translation $\tau$ such that $\pi \circ f = \tau \circ \pi$ (cf.~\eqref{eq_tau}), and 
\item[\textup{(iii)}] the regular set $\RR_{\pi}$ contains uncountably many elements.
\end{enumerate}
\end{defn}

The following simple but important observation plays a crucial role in the constructions below. In what follows, we denote $\OO_f(z)$ the full orbit of a point $z$ under $f$. 

\begin{lem}\label{lem_semi_conjugation}
Let $f \in \homeo_{\#}(\T^2)$, with $\pi$ the corresponding semi-conjugacy. 
Then $f$ has a unique minimal set $\MM$ and 
\begin{equation}
\MM = \Cl(\RR_{\pi}) = \Cl(\OO_f(z)),
\end{equation}
for any $z \in \RR_{\pi}$. 
\end{lem}

\begin{proof}
Let us first prove $\MM$ is the unique minimal set. Let $\MM$ and $\MM'$ be two minimal sets for $f$. Because $\MM$ is closed and $f$-invariant, $\pi(\MM)$ is closed and $\tau$-invariant. In particular, $\pi(\MM)$ contains the complete $\tau$-orbit of 
every point $\pi(z) \in \pi(\MM)$, where $z \in \MM$. Since every orbit of $\tau$ is dense, we have that $\pi(\MM) = \T^2$. Similarly, $ \pi(\MM') = \T^2$. Take $z \in \RR_{\pi} \neq \emptyset$. 
Then $\{z\} = \pi^{-1}(\pi(z))$ is contained in both $\MM$ and $\MM'$. As two minimal sets are either identical or disjoint, this implies that $\MM = \MM'$. Thus $\MM$ is unique.

Next we prove that $\MM = \Cl( \OO_f(z) )$, for every $z \in \RR_{\pi}$. Take any $z \in \RR_{\pi}$ and consider $\OO_f(z)$. Then $\Cl(\OO_f(z))$ is closed and invariant, hence it contains the (unique) minimal set $\MM$ of $f$, i.e. $\MM \subseteq \Cl( \OO_f(z) )$. 
We need to show that $\Cl( \OO_f(z) ) \subseteq \MM$. If $\MM \cap \OO_f(z) = \emptyset$, then as $\pi(\MM) = \T^2$ by the above, there exists a point $z' \in \MM$ for which $\pi(z') = \pi(z)$. However, this contradicts the assumption that $z \in \RR_{\pi}$
and thus $\pi^{-1}(\pi(z)) = \{z\}$. It follows that $\MM \cap \OO_f(z) \neq \emptyset$. Let $z' \in \MM \cap \OO_f(z)$, then $\OO_f(z') = \OO_f(z) \subseteq \MM$, since $\MM$ is invariant. But since $\MM$ is also closed, $\Cl(\OO_f(z)) \subseteq \MM$. Hence $\MM = \Cl(\OO_f(z))$, 
for any $z \in \RR_{\pi}$ and, consequently, $\MM = \Cl(\RR_{\pi})$.
\end{proof}

Let us further introduce the following notation, to be used in the proofs of Examples~\ref{ex_unbdd_disk},~\ref{ex_cylinders} and~\ref{ex_type_III} below. A non-transitive orientation preserving circle homeomorphism with irrational rotation number will be referred to as a {\em Denjoy counterexample}\index{Denjoy counterexample}. Moreover, given a Cantor set in the circle $\QQ = \T^1 \setminus \bigcup_{k \in \Z} I_k$, we denote $\QQ_{\tu{rat}} \subset \QQ$ and $\QQ_{\tu{irr}} = \QQ \setminus \QQ_{\tu{rat}}$ the {\em rational}\index{Cantor set!rational part} and {\em irrational}\index{Cantor set!irrational part} part of $\QQ$, comprised of all the endpoints of the deleted intervals and the complement in $\QQ$ of these endpoints respectively. It is readily verified, using Poincar\'e's Theorem, that 
\begin{enumerate}
\item[\textup{(1)}] a product of a Denjoy counterexample and an irrational rotation, and 
\item[\textup{(2)}] a product of two Denjoy counterexamples,
\end{enumerate}
provided the factors are chosen such that the corresponding rotation numbers are rationally independent, are elements of $\homeo_{\#}(\T^2)$.

\subsection{A topological blow-up procedure}\label{subsec_constr}

In order to construct the examples, we need a tool with which to construct homeomorphisms exhibiting the desired behaviour.
Below we devise such a tool that enables us to {\em blow up} an orbit of a point under a homeomorphism to a collection of disks. 
A. Bi\'s, H. Nakayama and P. Walczak in~\cite{bis_2} define such a blow-up procedure that works for (groups of) diffeomorphisms. J. Aarts and L. Oversteegen  in~\cite{aarts} defined a similar blowup construction for a homeomorphism of the annulus that preserves radial lines. 
In both constructions, this allows for the mapping to be extended to the disks glued to the surface by the infinitesimal behaviour of the mapping. As this would not work for a general homeomorphism, we circumvent this by inductively blowing up punctures to disks, by pulling back points near a puncture along leaves of a dynamically defined foliation emanating from the puncture. 
We use the continuity of the foliation to define an extension of the mapping to the disks. 

\medskip

Let $f \in \homeo_{\#}(\T^2)$, with $\pi \colon \T^2 \ra \T^2$ the semi-conjugacy between $f$ and an irrational translation $\tau$. 
Take a point $z_0 \in \RR_{\pi}$ and consider its orbit $\OO_{f}(z_0)$; we do not require $f$ to be minimal, so $\OO_{f}(z_0)$ may or may not be 
dense. Define $\Gamma = \T^2 \setminus \OO_{f}(z_0)$. Clearly, $f(\Gamma) = \Gamma$ and $f|_{\Gamma}$ is a homeomorphism. Let 
$B_{\delta} := B(z_0, \delta) \subset \T^2$ be the embedded closed Euclidean disk of radius $0 < \delta \leq 1/4$ centered at $z_0$. Choose 
$0 < \delta_0 \leq 1/4$ and let $\FF_0$ be the foliation of $B_{\delta_0}$ by straight rays emanating from $z_0$, see Figure~\ref{fig_blow_up}. The leaves $\rho_{\theta} \in \FF_0$ are parametrized by $\theta \in [0, 2 \pi)$. Fix $0 < \epsilon_0 < 1$. To blow up the punctures to disks, we define the following auxiliary planar map, which reads in polar coordinates,
\begin{equation} 
g_{\epsilon} \colon \widetilde{B}_1 \setminus \{0\} \ra \widetilde{B}_1 \setminus \widetilde{B}_{\epsilon}, \quad 
g_{\epsilon} (r, \theta) = \left( \frac{r + \epsilon}{1+ r \epsilon}, \theta \right)
\end{equation}
where $\widetilde{B}_{\rho} \subset \R^2$ is the closed Euclidean disk centered at $0 \in \R^2$ of radius $0 < \rho \leq 1$. Conjugating $g_{\epsilon_0}$ with a linear injection (into the torus) 
$\lambda_{0} \colon \widetilde{B}_{1} \hookrightarrow B_{\delta_0}$ yields a homeomorphism 
\begin{equation}\label{eq_constr_h_0}
h_0 = \lambda_{0} \circ g_{\epsilon_0} \circ \lambda_{0}^{-1} \colon A_0 \longrightarrow h_0(A_0),
\end{equation}
where $A_0 = B_{\delta_0} \setminus \{z_0\}$ and $h_0(A_0) = B_{\delta_0} \setminus B_{\epsilon_0 \delta_0}$ the corresponding annuli.
We can extend $h_0$ to $\T^2 \setminus \{z_0\}$ by declaring it to be the identity off $A_0$, the homeomorphism we denote again by $h_0$, 
and it naturally acts on $\Gamma \subset \T^2 \setminus \{z_0\}$ by restriction. Note that $h_0$ acts on $B_{\delta_0} \setminus \{z_0\}$ 
along the leaves of the foliation $\FF_0$. 

\begin{figure}[h]
\begin{center}
\psfrag{D_1}{$A_0$}
\psfrag{D_2}{$h_0(A_0)$}
\psfrag{D_e}{$\Sigma_0$}
\psfrag{rho_0}{$\rho_0$}
\psfrag{rho_theta}{$\rho_{\theta}$}
\psfrag{z_0}{$z_0$}
\psfrag{theta}{$\theta$}
\psfrag{h_0}{$h_0$}
\includegraphics[scale=0.8]{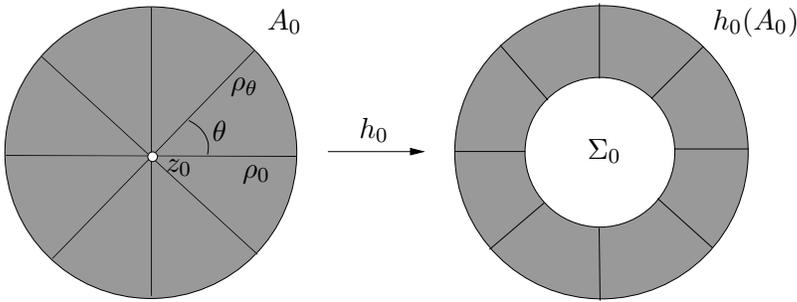}
\caption[Blow-up of the puncture $z_0$]{Radial blow up of a puncture to a disk.}
\label{fig_blow_up}
\end{center}
\end{figure}

Define $\Gamma_0 = h_0(\Gamma)$ and define the homeomorphism 
\begin{equation}
f_0 \colon \Gamma_0 \ra \Gamma_0, \quad f_0 = h_0 \circ f|_{\Gamma} \circ h_0^{-1}
\end{equation}
and define the continuous $\phi_0 \colon \Gamma_0 \ra \Gamma$ where $\phi_0 = h_0^{-1}$. Note that, by construction, 
$f_0 = \phi_0^{-1} \circ f|_{\Gamma} \circ \phi_0$. Define $\Sigma_0 := \Int(B_{\epsilon_0 \delta_0})$ and $\gamma_0 := \partial \Sigma_0$, 
see again Figure~\ref{fig_blow_up}. Consider the points $z_{\pm 1} := \phi_0^{-1}(f^{\pm 1}(z_0))$ and define
\begin{equation}\label{eq_blow_up_d1}
d_1 = \frac{1}{4} \min \left\{ (1/4)^2, d(z_{-1}, z_1), d(z_{-1}, \Sigma_0 ), d(z_{1}, \Sigma_0) \right\} > 0.
\end{equation}
Given $0 < \epsilon_0 < \epsilon < 1$, define $\epsilon' = \frac{\epsilon + \epsilon_0}{2}$ and define the second auxiliary planar map
\begin{equation} 
q_{\epsilon} \colon \widetilde{B}_{\epsilon} \setminus \widetilde{B}_{\epsilon_0} \ra \widetilde{B}_{\epsilon} \setminus \widetilde{B}_{\epsilon'}, \quad q_{\epsilon} = \hat{g}_{\epsilon' / \epsilon} \circ \hat{g}^{-1}_{\epsilon_0 / \epsilon},
\end{equation}
where $r_{\epsilon} : \widetilde{B}_{1} \ra \widetilde{B}_{\epsilon}$ is a linear (planar) rescaling, and
\begin{equation}
\hat{g}_{\delta / \epsilon} := r_{\epsilon} \circ g_{ \delta / \epsilon} \circ r_{\epsilon}^{-1},
\end{equation}
for $\epsilon_0 \leq \delta < \epsilon$. Let $\lambda_{\epsilon} \colon \widetilde{B}_{\epsilon} \hookrightarrow B_{\epsilon \delta_0}$ 
be the linear injection of the disk $\widetilde{B}_{\epsilon} \subset \R^2$ onto the disk $B_{\epsilon \delta_0} \subset \T^2$. 
Define $A_{\epsilon, \epsilon'} = B_{\epsilon' \delta_0} \setminus B_{\epsilon \delta_0}$ and 
\begin{equation}\label{eq_constr_k_eps}
\hat{q}_{\epsilon} \colon A_{\epsilon, \epsilon_0} \ra A_{\epsilon, \epsilon'}, \quad 
\hat{q}_{\epsilon} = \lambda_{\epsilon} \circ q_{\epsilon} \circ \lambda_{\epsilon}^{-1}.
\end{equation}
In words, $\hat{q}_{\epsilon}$ has the effect of mapping the annulus $A_{\epsilon, \epsilon_0}$ radially, i.e. along (part of) the foliation $\FF_0$, to the annulus $A_{\epsilon, \epsilon'}$ with the same outer boundary curve, but larger inner boundary curve, so as to half the modulus of the annulus. There exist $0 < \epsilon_0 < \epsilon_{\pm 1} <1$, such that, if we denote $A_{\pm 1} := f^{\pm 1}_0(A_{\epsilon_{\pm 1}, \epsilon_0})$, then $\diam(A_{\pm 1}) \leq d_1$. Define 
\begin{equation}\label{eq_constr_h_1}
h_{\pm 1} \colon A_{\pm 1} \ra h_{\pm 1}( A_{\pm 1} ), \quad h_{\pm 1} = f_0^{\mp 1} \circ \hat{q}_{\epsilon_{\pm 1}} \circ f_0^{\pm 1},
\end{equation}
defined on $A_{-1} \cup A_{1}$ and we extend $h_{\pm 1}$ to $\T^2$ by declaring it to be the identity off $A_{\pm 1}$. 
The annuli $A_{\pm 1}$ are foliated by $\FF_{\pm 1} = f_0|_{A_{\epsilon_{\pm 1}, \epsilon_0}} ( \FF_0 )$. 
The maps $h_{\pm 1}$ have the effect of blowing up the puncture $z_{\pm 1}$ along the foliation $\FF_{\pm 1}$ to a disk $\Sigma_{\pm 1}$, 
see Figure~\ref{fig_blow_up_foliation}. Denote $\Sigma_{\pm 1}$ the open disks obtained by blowing up the corresponding puncture 
$z_{\pm 1}$. As $\partial \Sigma_{\pm 1} = f_{0}( C_{\epsilon'_{\pm 1}} )$, where $C_{\epsilon'_{\pm 1}}$ is the Euclidean circle centered 
at $z_0$ of radius $\epsilon_0 < \epsilon'_{\pm 1} < \epsilon_{\pm 1}$, $\gamma_{\pm 1} := \partial \Sigma_{\pm 1}$ is a simple closed curve, as $f^{\pm 1}_{0}|_{A_{\epsilon_{\pm 1}, \epsilon_0}}$ is a homeomorphism.

\begin{figure}[h]
\begin{center}
\psfrag{D_1}{$A_1$}
\psfrag{D_2}{$h_1(A_1)$}
\psfrag{Sigma_1}{$\Sigma_1$}
\psfrag{rho'}{$\rho_{1, \theta}$}
\psfrag{z_1}{$z_1$}
\psfrag{f_1}{$h_1$}
\includegraphics[scale=0.8]{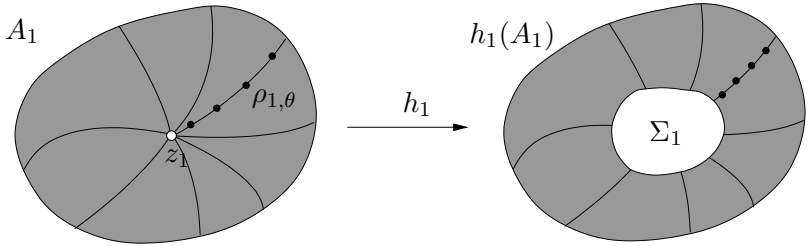}
\caption[Blow-up of the puncture $z_1$]{Blowing up the puncture $z_1$ to a disk; $\rho_{1, \theta}$ is a leaf of the foliation $\FF_1$ emanating from $z_1$ and $h_1$ has the effect of pulling back points on $\rho_{1, \theta} \in \FF_1$ along this leaf, for every $\theta \in [0,2 \pi)$.}
\label{fig_blow_up_foliation}
\end{center}
\end{figure}

Define $\hat{h}_1 := h_{-1} \circ h_1$ on $A_{-1} \cup A_{1}$, define $\Gamma_1 = \hat{h}_1 ( \Gamma_0 )$ 
and define the homeomorphism
\begin{equation}\label{eq_constr_f_1}
f_1 \colon \Gamma_1 \ra \Gamma_1, \quad f_1 := \hat{h}_1 \circ f_0 \circ \hat{h}_1^{-1}.
\end{equation}
Further, define the continuous $\phi_1 \colon \Gamma_1 \ra \Gamma$ where $\phi_1 = \phi_0 \circ \hat{h}_1^{-1}$.

\medskip

We proceed by induction. Assume we have blown up the punctures $z_{k}$ to disks $\Sigma_k$, where $-n+1 \leq k \leq n-1$,
and consider the points $z_{\pm n} := \phi_{n-1}^{-1}(f^{\pm n}(z_0))$. Define $\Delta_{n-1} = \bigcup_{k=-n+1}^{n-1} \Sigma_k$ and define
\begin{equation}\label{eq_constr_d_n}
d_{n} = \frac{1}{4} \min \left\{ (1/4)^{n+1}, d(z_{-n}, z_{n}), d(z_{-n}, \Delta_{n-1} ), d(z_{n}, \Delta_{n-1} ) \right\} > 0.
\end{equation}
There exist $0 < \epsilon_0 < \epsilon_{\pm n} < 1$, such that $\diam(A_{\pm n}) \leq d_n$. Define
\begin{equation}\label{eq_constr_h_n}
h_{\pm n} \colon A_{\pm n} \ra h_{\pm n}( A_{\pm n} ), \quad h_{\pm n} = f_{n-1}^{\mp n} \circ \hat{q}_{\epsilon_{\pm n}} \circ f_{n-1}^{\pm n},
\end{equation}
defined on $A_{\pm n}$, where we can extend $h_{\pm n}$ to $\T^2$ by declaring it to be the identity off $A_{\pm n}$. 
The annuli $A_{\pm 1}$ are foliated by $\FF_{\pm n} = f^{n}_{n-1}|_{A_{\epsilon_{\pm n}, \epsilon_0}} ( \FF_0 )$. 
The maps $h_{\pm n}$ blow up the puncture $z_{\pm n}$ along the foliation $\FF_{\pm n}$ to a disk $\Sigma_{\pm n}$. 
The boundaries $\gamma_{\pm n}$ are again simple closed curves, as $\gamma_{\pm n} = f^{\pm n}_{n-1}( C_{\epsilon'_{\pm n}} )$, 
where $C_{\epsilon'_{\pm n}}$ is the Euclidean circle centered at $z_0$ of radius $\epsilon_0 < \epsilon'_{\pm n} < \epsilon_{\pm n}$ and
$f^{\pm n}_{n-1}|_{A_{\epsilon_{\pm n}, \epsilon_0}}$ is a homeomorphism. Define $\hat{h}_n := h_{-n} \circ h_n$ on $A_{-n} \cup A_{n}$, 
define $\Gamma_n = \hat{h}_n ( \Gamma_{n-1} )$ and define the homeomorphism
\begin{equation}\label{eq_constr_f_n}
f_n \colon \Gamma_n \ra \Gamma_n, \quad f_n := \hat{h}_n \circ f_{n-1} \circ \hat{h}_n^{-1}.
\end{equation}
Further, define the continuous $\phi_n \colon \Gamma_n \ra \Gamma$ where $\phi_n = \phi_{n-1} \circ \hat{h}_n^{-1}$.

\medskip

Next, we show that the above sequences of maps and homeomorphisms converge and have the desired properties. First, by~\eqref{eq_constr_d_n} combined with~\eqref{eq_constr_h_n}, it holds that $\Gamma_n \subset \Gamma_{n-1}$ and that $\Gamma_{\infty} = \lim_{n \ra \infty} \Gamma_n$ converges in the Hausdorff sense, as $\sum_{n \geq 0} 1/4^{n+1} < 1 < \infty$. Denote $\NN = \Cl(\Gamma_{\infty})$. Notice that 
$\NN = \Gamma_{\infty} \cup \bigcup_{k \in \Z} \gamma_k$, since no point in $\Sigma_k$ can be the limit point of points in $\Gamma_{\infty}$ as $\Sigma_k \cap \Gamma_{\infty} = \emptyset$. Furthermore, note that, as the boundary curves $\gamma_k$ are simple closed curves, the extension $\bar{\phi}_n \colon \Cl(\Gamma_n) \ra \T^2$ of $\phi_n$ is continuous. 

\begin{lem}\label{lem_constr_converge}
The homeomorphisms $f_n \colon \Gamma_n \ra \Gamma_n$ converge to a homeomorphism $f_{\infty} \colon \Gamma_{\infty} \ra \Gamma_{\infty}$,
and extends to a homeomorphism $f' \in \homeo_*(\T^2)$ with $f'(\NN) = \NN$. Further, the disks $\{ \Sigma_k \}$ in the complement of $\NN$ 
are interiors of closed topological disks. Similarly, the continuous maps $\phi_n$ converge to a continuous map $\phi_{\infty} \colon \Gamma_{\infty} \ra \Gamma$, and extends to a continuous $\phi \colon \T^2 \ra \T^2$ for which $\phi(\NN) = \T^2$. Furthermore, $f'$ is semi-conjugate to $f$ 
through $\phi$.
\end{lem}

\begin{proof}
First, we show that $f_n \ra f_{\infty}$ converges to a homeomorphism of $\Gamma_{\infty}$. Indeed, for every $n \geq 0$,
$f_n \colon \Gamma_n \ra \Gamma_n$ is a homeomorphism and we observed above that $\Gamma_n \ra \Gamma_{\infty}$ converges. 
As $\hat{h}_n$ moves points by no more than the distance of $d_n \leq 1/4^{n+1}$, and $\sum_{n \geq 0} 1/4^{n+1} < 1 < \infty$, $f_n \ra f_{\infty}$ converges uniformly and thus the limit $f_{\infty}$ is a homeomorphism. Further, we observed that $\gamma_k$ is a simple closed curve, 
for every $k \in \Z$ and thus $\Sigma_k$ is the interior of the closed topological disk $\Cl(\Sigma_k) = \Sigma_k \cup \gamma_k$. 

Next, we show that $f_{\infty}$ extends to a homeomorphism $f'$ of $\NN$. To this end, we first show that $f_{\infty}$ induces a homeomorphism from
$\gamma_k$ to $\gamma_{k+1}$, for every $k \in \Z$. To prove this, we note that the disks $\Sigma_k$, for $-n \leq k \leq n$, which have been constructed after $n$ steps, are left unmoved by future perturbations by virtue of our choice of $d_n$. Moreover, again by our choice of $d_n$, it holds that $f_n|_{\gamma_k} = f_{\infty}|_{\gamma_k} \colon \gamma_k \ra \gamma_{k+1}$, for $-n \leq k \leq n-1$, where $f_n|_{\gamma_k}$ and $f_{\infty}|_{\gamma_k}$ are the extensions of $f_n$ and $f_{\infty}$ to $\gamma_k$. To prove that ${f_n}_{|_{\gamma_k}}$ is a homeomorphism, it suffices to show that ${f_n}_{|_{\gamma_k}}$ is one-to-one and continuous. We prove this by induction, where we consider $0 \leq k \leq n$, 
the case for negative $k$ being handled by considering the inverse. 

Assume that after step $n-1$, we have shown that $f_{n-1}|_{\gamma_k} \colon \gamma_k \ra \gamma_{k+1}$, for $0 \leq k \leq n-2$ are homeomorphisms and consider step $n$, where we have to show that $f_{n}|_{\gamma_{n-1}} \colon \gamma_{n-1} \ra \gamma_{n}$ is a homeomorphism.
By choice of $\epsilon_n$, $A_n$ is disjoint from the previously constructed disks and disjoint from $A_{-n}$. Restricting to a smaller neighbourhood of $A_{n-1}$ if necessary, we may as well assume that $A_{n-1} \cap A_n = \emptyset$. As $h_n = \hat{h}_n |_{A_n}$ (as defined by~\eqref{eq_constr_h_1}) acts along the foliation $\FF_n$, $f_n$ sends leaves of $\FF_{n-1}$ to leaves of $\FF_n$ which foliate $A_{n-1}$ and $A_n$ respectively. To each $\theta \in [0,2\pi)$ corresponds a unique point $z(\theta) \in \gamma_{n-1}$ lying on $\rho_{n-1,\theta} \in \FF_{n-1}$, which is by $f_n$ mapped to a unique point $z'(\theta) = f_n(z(\theta)) \in \gamma_n$ lying on $\rho_{n, \theta} = f_n(\rho_{n-1, \theta})$. As these foliations are continuous, and as the curves $\gamma_{n-1}$ and $\gamma_n$ are (continuous) simple closed curves, the points $z'(\theta)$ vary continuously as $\theta$ varies, and thus continuously as $z(\theta)$ varies and this is what we needed to show. 

By induction, $f_{\infty}$ extends homeomorphically to every boundary curve $\gamma_k$, $k \in \Z$. It thus follows that the extension $f'$ to 
$\NN$ is one-to-one, as $\NN = \Gamma_{\infty} \cup \bigcup_{k \in \Z} \gamma_k$. To show $f'$ is continuous, we distinguish between two cases.
First, let $z \in \Gamma_{\infty}$. As $f_{\infty}$ is a homeomorphism, given a neighbourhood $V \subset \NN$ containing $z'$, we can find a 
small neighbourhood $U \subset \NN$, containing the point $z$ for which $f'(z) = z'$, such that $\Cl(f_{\infty}(W)) \subset V$, where $W = U \cap \Gamma_{\infty}$. As $\Cl(\Gamma_{\infty}) = \NN$ and $f'$ extends homeomorphically to $\NN$, we have that $f'(U) = f'(\Cl(W)) = \Cl( f_{\infty}(W)) \subset V$. Secondly, suppose that $z \in \gamma_k$ for some $k \in \Z$. For $N \geq k+1$, it holds that $f_N|_{\gamma_k} = f_{\infty}|_{\gamma_k} \colon \gamma_k \ra \gamma_{k+1}$ is a homeomorphism. Therefore, given a neighbourhood $V \subset \NN$ containing $z' = f_N(z) = f'(z)$, there exists a small neighbourhood $U \ni z$, such that $f_N(U) \subset V$. Choosing $N$ larger, and a smaller neighbourhood $U' \subset U$ containing $z$, if necessary, as $\sum_{n \geq N} d_n \ra 0$ for $N \ra \infty$, we have that $f'(U') \subset V$ as well. Thus $f'$ is continuous, and therefore a homeomorphism, being one-to-one as well. We can extend the homeomorphism $f' \colon \NN \ra \NN$ to a homeomorphism of $\T^2$ by extending, e.g. by Alexander's trick, the induced homeomorphisms of the boundary curves $\gamma_k$ to homeomorphisms of the corresponding closed disks $\Cl(\Sigma_k) = \Sigma_k \cup \gamma_k$. As the disks $\Sigma_k$ are disjoint, and $\diam(\Sigma_k)$ forms a null-sequence, the extension of $f' \colon \NN \ra \NN$ to $\T^2$ is a homeomorphism, which we denote again by $f'$. 

To show that $\phi \colon \NN \ra \T^2$ is continuous, we recall that $\Cl(\Gamma_n) = \T^2 \setminus \Delta_n$, where $n \geq 0$. 
As we observed, for every $n \geq 0$, $\phi_n \colon \Gamma_n \ra \Gamma$ is continuous and it extends to a continuous 
$\bar{\phi}_n \colon \Cl(\Gamma_n) \ra \T^2$. As $\phi_n = \phi_{n-1} \circ \hat{h}_n^{-1}$, with $\hat{h}_n$ as in~\eqref{eq_constr_h_n}, 
whose norm is bounded by $d_n$, $\phi = \lim_{n \ra \infty} \bar{\phi}_n \colon \NN \ra \T^2$ is continuous as a limit of uniformly converging continuous maps $\bar{\phi}_n$. By declaring $\phi(\Sigma_k) = f^k(z_0)$, $\phi$ extends to a continuous map defined on $\T^2$.

Finally, we show that $f' \in \homeo_{\#}(\T^2)$. First, we observe that, as $\RR_{\pi}$ is uncountable, and a countable number of points
of $\RR_{\pi}$ is blown up to disks, we have that $\RR_{\pi'}$ is uncountable. Further, the $\phi$ thus constructed is homotopic to the identity.
Thus it suffices to show that $\phi \circ f' = f \circ \phi$. For this, we note that for every $n \geq 0$ we have that 
$\phi_n \circ f_n = f|_{\Gamma} \circ \phi_n$, where $f_n \colon \Gamma_n \ra \Gamma_n$ and $\phi_n \colon \Gamma_n \ra \Gamma$. 
As both $f_n$ and $\phi_n$ converge uniformly and extend continuously to $\NN$, it follows that $\phi \circ f' = f \circ \phi$, 
where $f' \colon \NN \ra \NN$. Further, as $\Sigma_k$, along with $\gamma_k$, is mapped to a single point by $\phi$, it thus also holds 
that $\phi \circ f' = f \circ \phi$ when $f'$ is extended to $\T^2$.
\end{proof}

The following lemma, which combines Lemma~\ref{lem_semi_conjugation} and Lemma~\ref{lem_constr_converge} is the key ingredient in the 
construction of Examples~\ref{ex_cylinders} and~\ref{ex_type_III}. Let $B_{\delta_0} \setminus \{ z_0\} \subset \T^2$ be an embedded punctured disk centered at $z_0 \in \RR_{\pi}$ with $\delta_0 \leq 1/4$ and $\FF_0$ the corresponding foliation of $B_{\delta} \setminus \{z_0\}$ by straight rays emanating from $z_0$, in the notation of the construction above. A {\em wedge}\index{wedge} $\WW(r, \theta_1, \theta_2) \subset B_{r\delta_0} \setminus \{ z_0\}$ is the region bounded by two leaves $\rho_{\theta_1}, \rho_{\theta_2} \in \FF_0$, where $0 < | \theta_1 - \theta_2 | < \pi$ and $0 < r \leq 1$.

\begin{lem}\label{lem_struct_min_sets}
In the construction above, let $f'$ be semi-conjugate to $f$ through $\phi$ by blowing up the orbit $\OO_f(z_0)$, with $z_0 \in \RR_{\pi}$, 
to disks whose interiors are $\Sigma_k$, and $\gamma_k = \partial \Sigma_k$, where $k \in \Z$. Let $\MM'$ be the minimal set of $f'$ and 
define $\pi' = \pi \circ \phi$. Then
\begin{enumerate}
\item[\tu{(1)}] $\MM' = \Cl(\RR_{\pi'}) = \Cl( \phi^{-1}(\RR_{\pi} \setminus \OO_{f}(z_0)))$,
\item[\tu{(2)}] $\gamma_k \subset \MM'$, for all $k \in \Z$, if for every $0 < r \leq 1$ and every $\theta_1, \theta_2 \in [0,2\pi)$, 
with $0< | \theta_1 - \theta_2| < \pi$, we have that $\WW(r, \theta_1, \theta_2) \cap ( \RR_{\pi} \setminus \OO_f(z_0) ) \neq \emptyset$.
\end{enumerate} 
\end{lem}

\begin{proof}
To prove (1), as $\RR_{\pi}$ is uncountable, $R^0_{\pi} \neq \emptyset$. As the points $f^k(z_0)$ are blown up to disks, i.e. 
$\phi^{-1}(f^k(z_0)) = \Cl(\Sigma_k)$, where $\Cl(\Sigma_k)$ is a closed topological disk, we have that $\RR_{\pi'} = \phi^{-1}( \RR_{\pi} \setminus \OO_{f}(z_0))$. By Lemma~\ref{lem_semi_conjugation}, we have that $\MM' = \Cl(\RR_{\pi'})$, and this proves (1).

To prove (2), define $\RR^n_{\pi} := \phi_n^{-1}( \RR_{\pi} \setminus \OO_f(z_0))$. First assume that $\gamma_0 \subset \Cl(R^0_{\pi})$. As the size of the perturbations $\hat{h}_n$, by virtue of our choice of $d_n$, converge to zero as the perturbations approach $\gamma_0$, it holds that $\gamma_0 \subset \Cl(R^n_{\pi})$ for every $n \geq 0$. As the maps $\phi_n$ converge, it thus holds that $\gamma_0 \subset \Cl( \RR_{\pi'} ) = \MM'$, by (1). As $\MM'$ is $f'$-invariant, and $f'(\gamma_k) = \gamma_{k+1}$, we have that $\gamma_k \subset \MM'$, for every $k \in \Z$. To finish the proof, suppose that $\WW(r, \theta_1, \theta_2) \cap ( \RR_{\pi} \setminus \OO_f(z_0) ) \neq \emptyset$ for every $ 0 < r \leq 1$ and 
every $\theta_1, \theta_2 \in [0,2\pi)$ for which $0< | \theta_1 - \theta_2| < \pi$. We have to show that $\gamma_0 \subset \Cl(R^0_{\pi})$. Suppose, to derive a contradiction, that $\gamma_0 \cap \Cl( R^0_{\pi} ) \neq \gamma_0$. As $\gamma_0 \cap \Cl( R^0_{\pi} )$ is closed, this implies there exists an open subarc $\eta \subset \gamma_0$ such that $\eta \cap \Cl( R^0_{\pi} ) = \emptyset$. Let $z \in \eta$ be the midpoint of $\eta$ and let $\eta' \subset \eta$ be a closed subsegment properly contained in $\eta$, and containing $z \in \eta$, with endpoints 
$\{ z^-, z^+ \} = \partial \eta'$. Let $\rho_{\theta_1}$ and $\rho_{\theta_2}$ be the two rays passing through $z^-$ and $z^+$ and $\WW(1, \theta_1, \theta_2)$ the corresponding wedge. As $\Cl( R^0_{\pi} )$ is closed, there exists an open neighbourhood $U \supset \eta'$ such that $U \cap \Cl( R^0_{\pi} ) = \emptyset$. However, this implies that $\WW(r, \theta_1, \theta_2) \cap ( \RR_{\pi} \setminus \OO_f(z_0) ) = \emptyset$ for $r > 0$ sufficiently small, contrary to our assumption.
\end{proof}

\subsection{Minimal sets of type I}\label{ex_type_I}

It is well-known that, given any Sierpi\'nski set $S \subset \T^2$, there exist a homeomorphism $f \in \homeo_*(\T^2)$ for which the 
minimal set $\MM = S$. The following example can be found in~\cite[Thm 3]{bis_1}. We will only sketch the proof.

\begin{example}[Type I : a quasi-Sierpi\'nski set~\cite{bis_1}]\label{ex_quasi_sierp}
There exist homeomorphisms $f \in \homeo_*(\T^2)$ for which the minimal set $\MM$ is a quasi-Sierpi\'nski set,
but not a Sierpi\'nski set.
\end{example}

\begin{figure}[h]
\begin{center}
\psfrag{Sigma}{$\Sigma$}
\psfrag{Sigma_1}{$\Sigma_1$}
\psfrag{Sigma_2}{$\Sigma_2$}
\psfrag{sim}{$\slash \sim$}
\psfrag{eta}{$\eta$}
\includegraphics[scale=0.5]{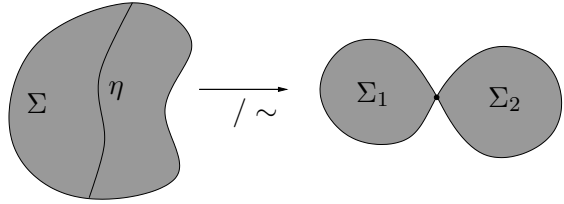}
\caption[A quasi-Sierpi\'nski set]{Construction of a quasi-Sierpi\'nski set: collapsing arcs to points.}
\label{fig_example_quasi_sierp}
\end{center}
\end{figure}

\begin{proof}[Sketch of the proof]
Let $\MM$ be a Sierpi\'nski minimal set of an $f \in \homeo_*(\T^2)$ and let $\Sigma$ be a component of the complement of $\MM$. 
Denote $\Sigma_n = f^n(\Sigma)$ and $\gamma_n = \partial \Sigma_n$. Take an arc $\eta \subset \Cl(\Sigma)$ such that only the endpoints of 
$\eta$ intersect $\gamma = \gamma_0$, see Figure~\ref{fig_example_quasi_sierp}. Let $\eta_n := f^n(\eta) \subset \Cl(\Sigma_n)$ the corresponding arcs in the image disks. Using techniques from decomposition theory, it can be shown that $\T^2 \slash \sim$, where $z \sim z'$ if and only 
if $z, z' \in \eta_n$ (i.e. collapsing the arcs $\eta_n$ to points), yields a well-defined quotient space homeomorphic to $\T^2$ and that $\MM$ quotients to a quasi-Sierpi\'nski set $\MM' = \MM \slash \sim$, which is not a Sierpi\'nski set. The corresponding quotient homeomorphism $f' \in \homeo_*(\T^2)$ has $\MM'$ as its minimal set and this minimal set $\MM'$ is locally connected.
\end{proof}

Next, we give an example of a minimal set which is of type I, but not locally connected. It shows the existence of homeomorphisms $f \in \homeo_*(\T^2)$ for which the minimal set $\MM$ is such that the complement consists of a single unbounded disk $\Sigma$, which
is $f$-invariant.

\begin{example}[Type I : unbounded disks]\label{ex_unbdd_disk}
There exist minimal sets $\MM$ of homeomorphisms $f \in \homeo_*(\T^2)$ of type I such that the complement of $\MM$ in $\T^2$ is a single unbounded disk.
\end{example}

The example uses a {\em derived-from-Anosov} type argument\footnote{See for example~\cite[Chapter 4]{pm} for these and related constructions.}, constructed initially by P. McSwiggen~\cite{mcswiggen} on the $3$-torus and subsequently generalized to the general $(k+1)$-torus, 
for $k \geq 1$ in~\cite{mcswiggen2}. For completeness, we recall this discussion for the case $k=1$ and collect the necessary results needed for our construction. Almost all of this material, up to the proof of Example 2, is taken {\em verbatim} from~\cite{mcswiggen2}.

Start with a hyperbolic $A \in \SL(2, \Z)$ and let $g_0 \colon \T^2 \ra \T^2$ be the induced linear toral automorphism. Let $z_0 \in \T^2$, with $z_0 = p(0)$ where $0 \in \R^2$ is the origin and $p \colon \R^2 \ra \T^2$ is the canonical projection, be the fixed point of $g_0$. Let $v^s, v^u \in \R^2$ be the stable and unstable eigenvector of $A$ respectively and $\lambda_s, \lambda_u$ the corresponding eigenvalues. Let $\FF^s_{\tu{lin}}, \FF^u_{\tu{lin}}$ be the stable and unstable foliations respectively relative to $g_0$ of $\T^2$ by parallel lines and let $\ell_0 \subset \T^2$ be the unstable leaf of $\FF^u_{\tu{lin}}$ passing through the saddle fixed point $z_0$. Relative to the standard basis, the eigenvectors $v^s, v^u$ have irrational slope, and, consequently, every leaf of $\FF^s_{\tu{lin}}$ and $\FF^u_{\tu{lin}}$ is an isometric immersion 
into $\T^2$ of a copy of $\R$. For future reference, define $\FF_{\tu{hor}}$ be the foliation of $\T^2$ by horizontal (relative to the standard basis) simple closed curves, parametrized by $y \in \T^1$; i.e. $\FF_{\tu{hor}} = \{ C_{y} \}_{y \in \T^1}$ with $C_y \subset \T^2$ the curve of height $y \in \T^1$.

Next we perturb $g_0$ on a small neighbourhood $U \ni z_0$ of the original saddle fixed point of $g_0$ to turn it into a repeller and create two additional saddle fixed points $z_{-1}, z_{1}$ close to $z_0$ through a pitch-fork bifurcation,
resulting in a DA diffeomorphism. We first describe the perturbation and subsequently we scale it into the neighbourhood $U$ of $z_0$. For this, we use the following, see~\cite[Lemmas 1.3 and 1.4]{mcswiggen2}.

\begin{lem}\label{lem_lambda}
Fix $\rho > 0$. Given any $\upsilon > \rho$ and any arbitrarily small $\delta >0$, for each $\kappa$ with $\rho > \kappa > 0$, there exists a $C^{\infty}$ function $\lambda_{\kappa} \colon \R \ra \R$ satisfying the following conditions.
\begin{enumerate}
\item[\tu{(i)}] $\lambda_{\kappa}(t) \equiv \upsilon$ for $t \leq 0$ and $\lambda_{\kappa}(t) \equiv \kappa$ for $t \geq 1$
\item[\tu{(ii)}] $- \delta \leq t \lambda_{\kappa}'(t) \leq 0$ for all $t$ and $\lambda_{\kappa}'(t) < 0$ if $\upsilon > \lambda_{\kappa}(t) > \kappa$. In addition, if $\kappa \geq \kappa'$, then $0 \leq \lambda_{\kappa}(t) - \lambda_{\kappa'}(t) \leq \kappa - \kappa'$
for all $t$.
\item[\tu{(iii)}] There exists $t_0 >0$, depending only on $\rho$ and $\delta$, where $\lambda_{\kappa}(t_0) = \rho$ and such that $\lambda_{\kappa}(t)$ is independent of $\kappa$ for $t \leq t_0$ and independent of $\upsilon$ for $t \geq t_0$.
\end{enumerate}
Further, if we define $\eta(t) = t \lambda(t)$, then $0 \leq (\lambda - \eta')(t) \leq \delta$ and $0 \leq \eta(t) - \kappa t \leq \delta$ for $t \geq 0$.
\end{lem}

Fix $\upsilon$ with $\lambda_u > \upsilon > 1$. Choose $\rho$ with $1 > \rho > \lambda_s$ and $\delta >0$ small enough so that $1-\delta > \rho$.
Choose $\bar{\upsilon}$ satisfying $1- \delta \geq \bar{\upsilon} > \rho$. Let $\lambda, \bar{\lambda}$ be the functions of Lemma~\ref{lem_lambda} relative to the 
choices of $\kappa = \lambda_s$ and these $\upsilon$ and $\bar{\upsilon}$ respectively. Denote $E_1 \oplus E_2$ the splitting of the tangent bundle $T\R^2$ into the unstable and stable direction relative to $A \in \SL(2, \Z)$. By a linear change of coordinates $B$ taking $E_1 \oplus E_2$ to $\R^2$, the action reads $B A B^{-1}(x,y) = (\lambda_u x, \lambda_s y)$. Let $\chi_0, \chi_1, \chi_2 \colon \R \ra [0,1]$ be even $C^{\infty}$ functions satisfying $\chi_0 + \chi_1+ \chi_2 \equiv 1$, $\chi_0(0) = 1$; $\chi_0(t) \equiv 0$ for $|t| \geq 1$; $\chi_2(t) \equiv 0$ for $|t| \leq 1$ and $\chi_2(t) \equiv 1$ for $|t| \geq 2$. We may further assume that $|\chi'_i| \leq 2$ for all $i$. 

Define $G(x,y) = (\lambda_u x, L(x,y) y)$, where $L(x,y) = \chi_0 \lambda(y) + \chi_1(x) \bar{\lambda}(y) + \chi_2 \lambda_s$. The derivative of $G$ has the following form
\begin{equation} 
DG_{(x,y)} = \left( \begin{matrix} \lambda_u & 0 \\ c(x,y) & N(x,y) \end{matrix} \right), 
\end{equation}
where 
\begin{equation}
c(x,y) = \left \{ \begin{matrix} \chi_0'(x) ( \lambda(y) - \bar{\lambda}(y))y & |x| \leq 1 \\ \chi_1'(x) ( \lambda(y) - \lambda_s)y & 1 \leq |x| \leq 2 \end{matrix} \right.
\end{equation}
and $N(x,y) = L(x,y) + (\chi_0(x) \lambda'(y) + \chi_1(x) \bar{\lambda}'(y))y$. The following estimates hold, cf.~\cite[Lemma 2.2]{mcswiggen2}.

\begin{lem}\label{lem_diff_est}
We have that $| c(x,y) | \leq 2 \delta$ and $\lambda_s \leq |N(x,y)| \leq \upsilon$. Moreover, if $|y| \geq t_0$ or $|x| \geq 1$, then $N(x,y) < 1$.
\end{lem}

\begin{proof}
Set $t = |y|$ and $\hat{t} = y/|y|$. For $|x| \leq 1$ $c(x,y) = \chi_0'(x)(\lambda(t)t - \bar{\lambda}(t)t) \hat{t}$. As $\lambda(t) t - \bar{\lambda}(t) t = \eta(t) - \bar{\eta}(t) \leq \eta(t) - \lambda_st \leq \delta$ by 
Lemma~\ref{lem_lambda}, it holds that $|c(x,y)| \leq \delta |\chi_0'(x)| \leq 2 \delta$ by choice of $\chi_0$. For $1 \leq |x| \leq 2$, $\lambda(t) - \bar{\lambda}(t)$ is replaced by $\bar{\lambda}(t)t - \lambda_s t = \bar{\eta}(t) - \lambda_s$ 
which has the same estimate. For $|x| \geq 2$, we have that $c(x,y) = 0$. 

To estimate $N(x,y)$, we first estimate $(\chi_0(x) \lambda'(y) + \chi_1(x) \bar{\lambda}'(y))y$. As $|y \lambda'(y)| \leq \delta$ and $|y \bar{\lambda}'(y)| \leq \delta$ by Lemma~\ref{lem_lambda}, 
we have that $|(\chi_0(x) \lambda'(y) + \chi_1(x) \bar{\lambda}'(y))y| \leq \delta$. As $\lambda(t), \bar{\lambda}(t) \geq \lambda_s$, we have that $|N(x,y)| \geq \lambda_s$.
For $t \geq t_0$, $\lambda(t) = \bar{\lambda}(t) \leq \rho$, again by Lemma~\ref{lem_lambda}. Therefore, $|N(x,y)| \leq \rho + \delta < 1$. If $|x| \geq 1$, we have that $|N(x,y)| \leq \bar{\lambda}(t) + \delta \leq \bar{\upsilon} + \delta < 1$.
Therefore, $|N(x,y)| < 1$ for $|y| \geq t_0$ or $|x| \geq 1$. Finally, if $t \leq t_0$ and $|x| \leq 1$, we have that $\lambda(t) \geq \bar{\lambda}(t)$. We have that $|N(x,y)| = | \chi_0 (\lambda(t) + t \lambda'(t)) + \chi_1( \bar{\lambda}(t) + t \bar{\lambda}'(t)) |$
As $\eta'(t) \leq \lambda(t)$ and $\bar{\eta}'(t) \leq \bar{\lambda}(t) \leq \lambda(t)$ by Lemma~\ref{lem_lambda}, and as $\eta'(t) = \lambda(t) + t \lambda'(t)$ and $\bar{\eta}'(t) = \bar{\lambda}(t) + t \bar{\lambda}'(t)$ by definition, it follows that
$|N(x,y)| = | \chi_0 \eta'(t) + \chi_1 \bar{\eta}'(t) | \leq \lambda(t) \leq \upsilon$ for all $t$. This finishes the proof.
\end{proof}

In particular, the map $G \colon \R^2 \ra \R^2$ is a $C^{\infty}$ diffeomorphism agreeing with $B A B^{-1}$ outside the square $\{(x,y) ~|~ |x| \leq 2,~ |y| \leq 1 \}$ and, as $N(0,0) = \upsilon > 1$, the origin is turned into a repeller. 
We can now rescale the perturbation so that the support is contained in $U \ni z_0$, the fixed point of the original linear $g_0$, and we obtain a perturbed system $g \colon \T^2 \ra \T^2$. Since rescaling does not affect estimates on first derivatives, 
relative to the splitting $E_1 \oplus E_2$, the derivative of $g$ has the form 
\begin{equation}
Dg_z = \left( \begin{matrix} \lambda_u & 0 \\ \widetilde{c}(z) & \widetilde{N}(z) \end{matrix} \right), 
\end{equation}
where $|\widetilde{c}(z)| \leq 2 \delta$ and $\lambda_s \leq |\widetilde{N}(z)| \leq \upsilon$ for all $z \in \T^2$. With $g$ defined as such, $Dg$ leaves the line bundle $E_2$ invariant and, although it is not a contraction on $E_2$, it expands in this direction by no more 
than $\upsilon < \lambda_u$. Set $E^{ps} := E_2$. Even though $E_1$ is no longer invariant, there is a new invariant line bundle $E^u$ everywhere close to $E_1$ which $Dg$ expands uniformly by $\lambda_u$ in a suitable metric and thus $E^u$ is a strong unstable 
direction. To see this, let $s \colon \LL(E_1, E_2) \ra \T^2$ be the standard vector bundle over $\T^2$ with fibers $s^{-1}(z) \colon \lin(E_1, E_2)$, where $\lin(E_1, E_2)$ is the space of linear maps from $E_1$ to $E_2$. A continuous section of $\LL(E_1, E_2)$ 
corresponds to a continuous line field on $\T^2$. If $S_z \in \lin(E_1(z), E_2(z))$, then the induced action of $Dg_z$ is given by $S_z \mapsto \widetilde{c}(z) \lambda_u^{-1} + \widetilde{N}(z) S_z \lambda_u^{-1} \in \lin(E_1(g(z), E_2(g(z))$. 
As $|\widetilde{N}(z) \lambda_u^{-1}| \leq \upsilon \lambda_u^{-1} < 1$, $Dg$ contracts the section, hence there is a unique invariant section $\mathcal{S} \in \LL(E_1, E_2)$. If we define $E^u$ to be the corresponding invariant line field for $Dg$, 
then $E^u \oplus E^{ps}$ is the continuous invariant splitting under $Dg$ of $\T^2$. As the image of the zero section is bounded in norm by 
\[ \sup_{z \in \T^2} \widetilde{c}(z) \lambda_u^{-1} \leq 2 \delta \lambda_u^{-1}, \] 
$Dg$ takes the disk of radius $r := 2 \delta (\lambda_u - \upsilon)^{-1}$ into itself, which implies $E^u(z)$ is contained in the cone $\{ (u,v) \in T_z \T^2 ~|~ |v| \leq r |u| \}$ for each $z \in \T^2$.
Choose $\delta >0$ so that $r \leq 1$. Defining an equivalent metric $\| (u,v) \|$ on $\T^2$ by declaring that $\|(u,v)\| = \max\{ |u|, |v| \}$ for $(u,v) \in E_1 \oplus E_2$, then under this metric $\|(u,v)\|= |u|$ for $(u,v) \in E^u$.
As $E^u$ is invariant, this implies that $Dg_z(u,v) \in E^u$ and thus $\| Dg_z (u,v) \| =  | \lambda_u  u | = \lambda_u | u | =  \lambda_u \|(u,v)\|$. Therefore, under this metric, $Dg$ expands $E^u$ by $\lambda_u$. Thus, all of $\T^2$ is a pseudo-hyperbolic set for $g$ 
with strong unstable direction $E^u$.

The classical theory of (un)stable manifolds now gives that, through each point $z \in \T^2$, there is a unique strong unstable manifold, which is a $C^{\infty}$ immersed copy of $\R$. As two such manifolds are either disjoint or equal, the union of these leaves 
foliate the torus $\T^2$ and this foliation, denoted $\FF^u_{\tu{DA}}$, is continuous. If $r$ is sufficiently small, implying that $\delta$ is sufficiently small, $\FF^u_{\tu{DA}}$ will be everywhere transverse to the leaves of the foliation $\FF_{\tu{hor}}$. Let us summarize this.

\begin{lem}\label{lem_DA_1}
The diffeomorphism $g \colon \T^2 \ra \T^2$ admits an $g$-invariant strong unstable lamination (i.e. a $C^0$ foliation by smooth leaves) $\FF^u_{\tu{DA}}$ whose leaves are everywhere nearly parallel to the linear unstable foliation 
$\FF^u_{\tu{lin}}$ and hence everywhere transversal to the horizontal foliation $\FF_{\tu{hor}}$ of $\T^2$.
\end{lem}

Recall that $z_0 \in \T^2$ is the repelling fixed point of $g$. Let $K$ be the set of points at which $g$ does not constract $E_2$. By Lemma~\ref{lem_diff_est}, $K$ is contained in the rectangle described in local coordinates by 
$|x| \leq 1$ and $|y| \leq t_0$. Let $\{z_{-1}, z_1 \}$ be the two intersection points with $\partial K$ with the local pseudo-stable leaf of $\FF^{ps}_{\tu{DA}}$ passing through $z_0$, which in local coordinates corresponds to $\{ (0,y) ~|~ \lambda(y) = 1\}$.
Since $\lambda$ is strictly decreasing, there is a unique $r_0$ such that $\lambda(r_0) = 1$. Let $J \subset \T^2$ be the open arc contained in the pseudo-stable leaf of $\FF^{ps}_{\tu{DA}}$ passing through $z_0$ whose endpoints coincide with $z_{-1}$ and $z_1$. 
As $\lambda(r_0) = 1$, $z_{-1}, z_1$ are hyperbolic fixed points for $g$. Let $\Sigma \subset \T^2$ be the basin of repulsion of $z_0$. Clearly, $\Sigma$ is the domain obtained by taking the union of all unstable leaves of 
$\FF^u_{\tu{DA}}$ passing through points $z \in J$ and $\Sigma$ is bounded by the two unstable leaves $W^u_{\pm 1} \in \FF^u_{\tu{DA}}$ passing through $z_{-1}$ and $z_1$ respectively. Further note that $\Sigma$ is $g$-invariant.

In what follows, let $\{ I_y^i \}_{i \in \Z}$ be the collection of connected components of $C_y \cap \Sigma$. The following holds, see~\cite[Thm 2.4]{mcswiggen2} for a detailed proof. 

\begin{lem}\label{lem_DA_semi}
If we define $h_y , h'_y \colon C_y \ra C_y$ the holonomy homeomorphisms, i.e. the return map to $C_y$ along the unstable foliation, of the linear and perturbed system respectively, 
then for every $y \in \T^1$, there exists a semi-conjugacy $\pi_y \colon C_y \ra C_y$ between the perturbed and linear system, i.e. a continuous $\pi_y$ such that $\pi_y \circ h'_y = h_y \circ \pi_y$. 
Further, for every $y \in \T^1$, $\Sigma \cap C_y$ is dense in $C_y$ and $h'_y(\bar{I}_y^i) \cap \bar{I}_y^i = \emptyset$ for every $i \neq 0$, where $\bar{I}$ denotes the closure $\Cl(I)$ of the 
interval $I$ in $C_y$.
\end{lem}

Let us further observe the following. 

\begin{lem}\label{lem_DA_disk}
The domain $\Sigma$ is an unbounded disk which is dense in $\T^2$.
\end{lem}

\begin{proof}
By Lemma~\ref{lem_DA_semi}, $\Sigma \cap C_y$ is dense in $C_y$ for every $y \in \T^2$, thus $\Sigma \subset \T^2$ is dense in $\T^2$ as $\bigcup_{y \in \T^1} C_y = \T^2$. 
To show that $\Sigma$ is an unbounded disk, we first recall the following. By~\cite[Thm 2.4]{mcswiggen2}, $g$ is semi-conjugate to $g_0$, through a semi-conjugacy which is homotopic to the identify, taking unstable leaves of 
$\FF^u_{\tu{DA}}$ injectively onto unstable leaves of $\FF^u_{\tu{lin}}$. Consequently, the lifts of the leaves of $\FF^u_{\tu{DA}}$ are homotopic (and thus uniformly close), through the semi-conjugacy between $g$ and $g_0$, to the corresponding lifts of the leaves 
of $\FF^u_{\tu{lin}}$. As the latter leaves are clearly unbounded in the lift, so are the lifts of the leaves of $\FF^u_{\tu{DA}}$ and, in particular, $\widetilde{\Sigma}$ is unbounded, where $\widetilde{\Sigma}$ is a lift of $\Sigma$ to the cover $\R^2$.
Further, as $\widetilde{\Sigma}$ is foliated by leaves homeomorphic to $\R$ passing through the cross-section $\widetilde{J} \subset \widetilde{\Sigma}$, where $\widetilde{J}$ is a lift of the open interval $J \subset \Sigma$ as defined above, 
it is homeomorphic to $(0,1) \times \R$ and thus simply connected. 

To show that $\Sigma$ is a disk, we need to show that $\Sigma$ contains no essential simple closed curves. To derive a contradiction, suppose there exists an essential simple closed curve $\gamma \subset \Sigma$.
As $\gamma$ lifts to a curve $\widetilde{\gamma}$ which is homotopic to a curve of type $(p,q)$, where $p, q \in \Z$, $\widetilde{\gamma}$ is homotopic, and thus uniformly close, to a straight line with rational slope. However, as the unstable leaves $W^u_{\pm 1}$
have lifts that are homotopic to a leaf of $\FF^u_{\tu{lin}}$ whose slope is irrational, this implies that $\widetilde{\gamma}$ has to intersect a lift of $W^u_{1}$ implying that $\gamma \cap W^u_{1} \neq \emptyset$, 
a contradiction. Thus $\Sigma$ is an unbounded disk indeed.
\end{proof}

Let us now proceed to the construction of our example.

\begin{proof}[Proof of Example~\ref{ex_unbdd_disk}]
As $C_{y} \cap \Sigma$ is a dense union of disjoint intervals, we have that $\QQ_y = C_y \setminus (C_{y} \cap \Sigma)$, being the circle minus a dense union of disjoint intervals, is a Cantor set. Note that the endpoints $\QQ_{y, \tu{rat}}$ of $\QQ_y$ are exactly the set of points $C_y \cap \partial \Sigma$ and that $\partial \Sigma = \bigcup_{y \in \T^1} \QQ_{y, \tu{rat}}$. The complement in $\QQ_y$ of these endpoints are the irrational points, i.e. $\QQ_{y, \tu{irr}} = \QQ_y \setminus \QQ_{y, \tu{rat}}$. Further, as the slope of $v^u$ is irrational, there exists a suitable $\nu \in \R$ such that $\nu v^u = (\alpha, \beta)$ with $1, \alpha, \beta$ rationally independent. For example, if we take $g_0 \colon \T^2 \ra \T^2$ to be Arnold's cat map, where $v^u = (1, \frac{1 + \sqrt{5}}{2})$, we may choose $\nu = e$ so that $(\alpha, \beta)$ is an irrational vector as 
$1, \sqrt{5}, e$ are rationally independent. Let $\tau := \tau_{\alpha, \beta}$ be the corresponding irrational translation of the torus where, by construction, $\tau(\ell_0) = \ell_0$.

Choose compatible orientations on the foliations $\FF^u_{\tu{DA}}$ and $\FF^u_{\tu{lin}}$. Given $y \in \T^1$, let $h_y , h'_y \colon C_y \ra C_y$ be the holonomy homeomorphisms, i.e. the return map to $C_y$ of the unstable foliation of the linear and perturbed system respectively. 
For every $y \in \T^1$, there exists a semi-conjugacy $\pi_y \colon C_y \ra C_y$ such that $\pi_y \circ h'_y = h_y \circ \pi_y$, by Lemma~\ref{lem_DA_semi}. As the continuous foliations $\FF^u_{\tu{DA}}$ and $\FF^u_{\tu{lin}}$ are 
everywhere transversal to the horizontal foliation, the holonomy homeomorphisms $h'_y$, and consequently the maps $\pi_y$, depend continuously on $y \in \T^1$. In other words, as $\T^2 = \bigcup_{y \in \T^1} C_{y}$, this defines a continuous 
\begin{equation}
\pi \colon \T^2 \ra \T^2, \quad \pi(x,y) = ( \pi_y(x), y ),
\end{equation} 
that has the property that $\pi(\Sigma \cup W^u_{-1} \cup W^u_{1}) = \ell_0$, as $\pi_y ( \bar{I}_y^i ) \in C_y \cap \ell_0$, by Lemma~\ref{lem_diff_est} for every $i \in \Z$ and $y \in \T^1$. For every $y \in \T^1$, there exists a homeomorphism $f_y \colon C_y \ra C_{r_{\beta}(y)}$, defined by mapping the point $z \in C_y$ to the first intersection point of the unique leaf of $\FF^u_{\tu{DA}}$ through $z$ with $C_{r_{\beta}(y)}$ (along the positive direction of the leaf), where $r_{\beta} \colon \T^1 \ra \T^1$ is the irrational rotation with rotation number $\rho(r) = \beta \mod \Z$, see Figure~\ref{fig_unbdd_disk}.

\begin{figure}[h]
\begin{center}
\psfrag{Sigma}{$\Sigma$}
\psfrag{C_y}{$C_y$}
\psfrag{C_y'}{$C_{r_{\beta}(y)}$}
\psfrag{z}{$z$}
\psfrag{z'}{$z'$}
\includegraphics[scale=0.55]{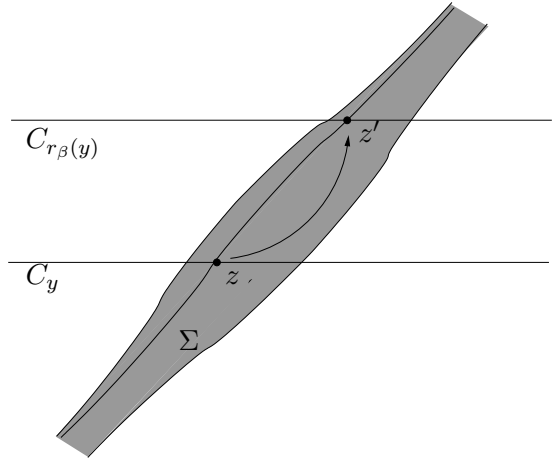}
\caption[The action of $f_y$]{$f_y$ maps the point $z \in C_y$ along a leaf of the foliation $\FF^u_{\tu{DA}}$ to a point $z' = f_y(z) \in C_{r_{\beta}(y)}$.}
\label{fig_unbdd_disk}
\end{center}
\end{figure}

As $\FF^u_{\tu{DA}}$ is a foliation which is transversal to $\FF^u_{\tu{hor}}$, $f_y$ is one-to-one. Further, as $\FF^u_{\tu{DA}}$ is continuous, $f_y$ is continuous as well and thus $f_y$ is a homeomorphism, for every $y \in \T^1$. Further, it follows from 
the definitions that $\pi_{r_{\beta}(y)} \circ f_y = \tau \circ \pi_y$. Define
\begin{equation}\label{eq_type_I_f}
f \colon \T^2 \ra \T^2, \quad f(x,y) = ( f_y(x) , r_{\beta}(y) ),
\end{equation}
As the maps $f_y$ are homeomorphisms for every $y \in \T^1$ and depend continuously on $y \in \T^1$ (by virtue again of the foliation $\FF_{\tu{DA}}^u$ being tranversal and continuous), it holds that $f$ as defined by~\eqref{eq_type_I_f} is a homeomorphism. 
It is clear that this $f$ is isotopic to the identity, $\pi$ is homotopic to the identity and, by construction, $\pi \circ f = \tau \circ \pi$ and $f(\Sigma) = \Sigma$ with $\Sigma$ the unbounded disk (by Lemma~\ref{lem_DA_disk}) bounded by the leaves $W^u_{\pm 1} \in \FF_{\tu{DA}}^u$, 
as $\tau(\ell_0) = \ell_0$. Let $\MM$ be the minimal set of $f$. As $\pi$ is one-to-one, except on $\Sigma \cup W^u_{-1} \cup W^u_{1}$ (again by Lemma~\ref{lem_diff_est}), it holds that 
\begin{equation}\label{eq_type_I_regular_1}
\RR_{\pi} = \pi^{-1}( \T^2 \setminus \ell_0 ) = \T^2 \setminus (\Sigma \cup W^u_{-1} \cup W^u_{1} ) = \bigcup_{y \in \T^1} \QQ_{y, \tu{irr}}.
\end{equation}
As $\QQ_y = \Cl(\QQ_{y, \tu{irr}})$, combining~\eqref{eq_type_I_regular_1} with Lemma~\ref{lem_semi_conjugation}, it follows that
\begin{equation}
\MM  = \Cl ( \bigcup_{y \in \T^1} \QQ_{y, \tu{irr}} ) = \bigcup_{y \in \T^1} \Cl \left( \QQ_{y, \tu{irr}} \right) 
= \bigcup_{y \in \T^1} \QQ_y = \T^2 \setminus \Sigma,
\end{equation}
where $\Cl ( \bigcup_{y \in \T^1} \QQ_{y, \tu{irr}} ) = \bigcup_{y \in \T^1} \Cl \left( \QQ_{y, \tu{irr}} \right)$ holds as the Cantor sets 
$\QQ_y$ (and therefore their irrational parts $\QQ_{y, \tu{irr}}$) depend continuously on $y \in \T^1$. Thus $\MM = \T^2 \setminus \Sigma$, with $\Sigma$ an unbounded and $f$-invariant disk, and $f \in \homeo_{\#}(\T^2)$, as required. This finishes the proof. 
\end{proof}

\subsection{Examples of type II}

Let us next give examples of homeomorphisms for which the connected components $\{ \Sigma_k \}$ of the complement of $\MM$ are essential annuli 
and disks.

\begin{example}[Type II : essential annuli and disks]\label{ex_cylinders}
There exist $f \in \homeo_*(\T^2)$ with minimal set $\MM$ of the form $\MM = \T^2 \setminus \bigcup_{n \in \Z} \Sigma_k$ 
with $\{ \Sigma_k \}_{k \in \Z}$ a collection of essential annuli and disks. Furthermore, the essential annuli can be constructed 
to have any characteristic $(p,q)$, where $\gcd(p,q) = 1$.
\end{example}

The proof of Example~\ref{ex_cylinders} uses the following. 

\begin{lem}\label{lem_anosov_translation}
Let $f \in \homeo_*(\T^2)$ and $f' = L_A^{-1} \circ f \circ L_A$ with $L_A \colon \T^2 \ra \T^2$ induced by a linear $A \in \SL(2,\Z)$.
Then $f' \in \homeo_*(\T^2)$.
\end{lem}

\begin{proof}
Let $L_A$ be a linear conjugation induced by an element $A \in \SL(2,\Z)$. As $f \in \homeo_0(\T^2)$, $f' \in \homeo_0(\T^2)$ as well. 
Let $F, F'$ a lift of $f, f'$ respectively. By~\cite[Lemma 2.4]{kocsard}, we have that 
\[ \rho(L_A^{-1} \circ F \circ L_A) = L_A^{-1} \rho(F)\mod \Z^2. \] 
Therefore, $\rho(f') = (\alpha', \beta')\mod \Z^2$, where 
\[ \alpha' = a \alpha + b \beta ~\textup{and}~\beta' = c \alpha + d \beta, \] 
with $a,b,c,d \in \Z$. The condition $N_1 + N_2 \alpha' + N_3 \beta' = 0$ implies that $N_1 = N_2 a + N_3 c = N_2b + N_3 d = 0$,
as $1, \alpha, \beta$ are rationally independent. Multiplying $N_2 a + N_3 c$ by $b$ and $N_2b + N_3 d $ by $a$ and subtracting yields that $N_2 ( ad -bc ) =0$, which yields that $N_2 = 0$ as $A \in \SL(2, \Z)$ and thus $ad-bc=1$. Similarly, it holds that $N_3=0$ and it thus follows that 
$1, \alpha', \beta'$ are rationally independent as well. Therefore, $f' \in \homeo_*(\T^2)$.
\end{proof}

\begin{proof}[Proof of Example~\ref{ex_cylinders}]
First, we construct a homeomorphism $f \in \homeo_{\#}(\T^2)$ for which the complement of $\MM$ is a collection of essential annuli of a given characteristic. Let $(p,q)$ with $\gcd(p,q) = 1$ be given. Let $f \in \homeo_{\#}(\T^2)$ be a product of a Denjoy counterexample $\varphi \in \homeo(\T^1)$ with irrational rotation number $\alpha \notin \Q$ and an irrational rotation $r_{\beta}$, with $\beta \notin \Q$ chosen so 
that $1, \alpha, \beta$ are rationally independent. The corresponding semi-conjugacy $\pi$ is of the form $\pi = ( \pi_1, \Id)$, where $\pi_1$ 
is the semi-conjugacy of $\varphi$ to the irrational rotation $r_{\alpha}$. As $\RR_{\pi} = \QQ_{\tu{irr}} \times \T^1$, the minimal set $\MM$ 
of $f$ is $\MM = \Cl(\QQ_{\tu{irr}} \times \T^1) = \QQ \times \T^1$, where $\QQ \subset \T^1$ is the Cantor minimal set of $\varphi$. The characteristic of the corresponding essential annuli is $(0,1)$. For later reference, denote $\{ \Sigma^a_t\}$ the collection of essential annuli in the complement of $\MM$. Given any pair $(p,q) \in \Z^2$ such that $\gcd(p,q) = 1$, there exists an element $A \in \SL(2,\Z)$ such that the (linear) $L_A \in \homeo(\T^2)$ induced by $A$ has the property that an essential simple closed curve of characteristic $(0,1)$ is mapped to an essential simple closed curve of characteristic $(p,q)$. By Lemma~\ref{lem_anosov_translation}, conjugating $f$ with $L_A$ gives a homeomorphism $f' \in \homeo_*(\T^2)$ and the components of the complement of the minimal set $\MM' = L_A^{-1}(\MM)$ now consists of essential annuli of characteristic $(p,q)$.

To obtain an example of a homeomorphism $f' \in \homeo_*(\T^2)$ with a minimal set $\MM'$ for which the complementary components contain both essential annuli and disks, we modify the above example as follows. Let $f$ again be the product homeomorphism given above. Choose $z_0 \in \RR_{\pi}$ and blow up the orbit $\OO_f(z_0)$ to disks by the procedure in section~\ref{subsec_constr}. This gives a homeomorphism $f' \in \homeo_{\#}(\T^2)$ and a continuous $\phi \colon \T^2 \ra \T^2$ such that $\phi \circ f' = f \circ \phi$ and we define $\pi' = \pi \circ \phi$.
We have that $\phi^{-1}(f^k(z_0)) = \Cl(\Sigma_k)$, where $\Sigma_k$ is the interior of the closed disk $\Cl(\Sigma_k)$ and $\gamma_k = \partial \Sigma_k$ a simple closed curve. Denote $\MM'$ the corresponding minimal set of $f'$. 

In order to show that the complement of $\MM'$ consists of essential annuli and disks, it suffices to show that $\gamma_k \subset \MM'$.
Indeed, as $\phi$ is one-to-one on $\T^2 \setminus \bigcup_{k \in \Z} \Cl(\Sigma_k)$, $\phi^{-1}(\Sigma^a_t)$ is again an essential annulus, 
for every $t \in \Z$, where $\phi^{-1}(\Sigma^a_t) \cap \MM' = \emptyset$. Thus to show that $\gamma_k \subset \MM'$, for every $k \in \Z$,
by Lemma~\ref{lem_struct_min_sets} (2), it suffices to show that for every $0 < r \leq 1$ and every $\theta_1, \theta_2 \in [0,2\pi)$, 
with $0< | \theta_1 - \theta_2| < \pi$, we have that $\WW(r, \theta_1, \theta_2) \cap (\RR_{\pi} \setminus \OO_f(z_0)) \neq \emptyset$, 
where we recall that the wedge $\WW(r, \theta_1, \theta_2) \subset B_{r\delta_0} \setminus \{ z_0\}$ is the region bounded by two leaves $\rho_{\theta_1}, \rho_{\theta_2} \in \FF_0$, 
where $0 < | \theta_1 - \theta_2 | < \pi$ and $0 < r \leq 1$ (see again section~\ref{subsec_constr}).

This is proved as follows: as $z_0 \in \RR_{\pi} = \QQ_{\tu{irr}} \times \T^1$, through every wedge $\WW(r, \theta_1, \theta_2)$ pass infinitely 
many vertical simple closed curves (i.e. the connected components of $\QQ_{\tu{irr}} \times \T^1$), arbitrarily close to $z_0$. 
As only countably many of these points are deleted from these curves, every wedge $\WW(r, \theta_1, \theta_2)$ contains points of $\RR_{\pi} \setminus \OO_f(z_0)$, for any $r>0$. Therefore, $\gamma_k \subset \MM'$ for every $k \in \Z$ indeed, where $\MM' = \Cl(\RR_{\pi'})$ 
by Lemma~\ref{lem_struct_min_sets} (1). This finishes the proof.
\end{proof}

\subsection{Examples of type III}\label{subsec_ex_III}

The most simple example of an extension of a Cantor set is of course a Cantor set itself. A homeomorphism $f \in \homeo_*(\T^2)$ admitting
such a Cantor minimal set is obtained by taking a product of two Denjoy-counterexamples with rationally independent rotation numbers. Its minimal set is the product of the Cantor minimal sets of its factors, and thus itself a Cantor set. Recall that an extension of a Cantor set $\MM$ is said to be non-trivial if not all connected components of $\MM$ are singletons. Below we give examples of non-trivial extensions of a Cantor set. Recently, F. B\'eguin, S. Crovisier, T. J\"ager and F. le Roux in~\cite[Thm 1.2]{beguin} also constructed an example of a homeomorphism $f \in \homeo_*(\T^2)$ for which the minimal set $\MM$ is a non-trivial extension of a Cantor set (in our terminology). This homeomorphism is constructed by adapting a quasiperiodically forced circle homeomorphism (see~\cite[Thm 1.2]{beguin}, cf. Counterexample~\ref{CA_1}) with a Cantor minimal set, and blowing up an orbit of points to arcs, using a construction due to M. Rees~\cite{rees_1, rees_2}. The minimal set thus constructed has a 
countable number of arcs among its connected components. The examples below show that in our class $\homeo_*(\T^2)$ there exist minimal sets,
which are non-trivial extensions of Cantor sets, which have separating connected components among its connected components. 

\begin{example}[Type III : extensions of Cantor sets]\label{ex_type_III}
There exist $f \in \homeo_*(\T^2)$ for which $\MM$ is a non-trivial extension of a Cantor set for which the non-degenerate components
are simple closed curves. 
\end{example}

\begin{rem}
In a way similar to~\cite[Thm 3]{bis_1}, cf. Example~\ref{ex_quasi_sierp}, by defining a suitable family of arcs in the disks enclosed by 
the non-degenerate components of the above minimal set, if we pass to the quotient by collapsing these arcs to points, this gives new quotient homeomorphisms of type III for which the corresponding minimal set is a again a non-trivial extension of a Cantor set, possible connected 
components of which include flowers and dendrites. 
\end{rem}

The proof of the above example needs two further lemmas. Let $\QQ \subset \T^1$ be a Cantor set and denote $\{ I_k \}_{k \in \Z}$ the collection of the connected components of $\T^1 \setminus \QQ$. In what follows, we denote $|I|$ the length (relative to the Haar measure on the circle) of an interval $I \subset \T^1$ and denote $\widetilde{d}_1$ the Euclidean metric on $\R$.

\begin{lem}\label{lem_cantor_set_perturb} 
There exist Cantor sets $\QQ \subset \T^1$ with the following property: there exists a point $x_0 \in \QQ_{\tu{irr}}$, an interval $J \subset \T^1$,
with $\partial J \subset \QQ_{\tu{irr}}$ and $x_0$ as midpoint of $J$, and a constant $C>0$, such that for every interval 
$I_k \subset J \setminus \QQ$, $k \geq 1$, it holds that $|I_k| \leq C (d(x_0, x_k))^2$, where $x_k$ is the midpoint of the interval $I_k$.
\end{lem}

\begin{proof}
First, let $[-1,1] \subset \R$ with midpoint $0 \in [-1,1]$. Inductively delete intervals from $[-1,1]$: at each step
$t \geq 1$, choose a point $x_t \in (-1,1) \setminus \bigcup_{s=0}^{t-1} I'_s$, with $x_t \neq 0$, and delete an interval 
$I'_t \subset (-1,1) \setminus \bigcup_{s=0}^{t-1} I'_s$, centered at $x_t$ and not overlapping $0$, of length at most $(\widetilde{d}_1(0,x_t))^2$. 
Repeating this ad infinitum produces a Cantor set $\QQ' \subset [-1,1]$. Given a Cantor set $\QQ \subset \T^1$, take a small (closed) interval $J \subset \T^1$ for which $\partial J \subset \QQ_{\tu{irr}}$. Replacing $J \cap \QQ \subset \T^1$ with a rescaled copy of $\QQ'$ into $J$ yields the desired Cantor set in $\T^1$, with $C = |J|/2$.
\end{proof}

\begin{lem}\label{lem_ex_III_cantor_blowup}
Let $f \in \homeo_{\#}(\T^2)$ be the product of two Denjoy counterexamples $\varphi, \psi \in \homeo(\T^1)$, semi-conjugate to an irrational translation $\tau$ through $\pi$. Let $\QQ_1, \QQ_2 \subset \T^1$ be two Cantor minimal sets of $\varphi$ and $\psi$ respectively, with $x_0 \in \QQ_{1, \tu{irr}}$ and $y_0 \in \QQ_{2, \tu{irr}}$ points and corresponding intervals $J_1$ and $J_2$ satisfying the 
conditions of Lemma~\ref{lem_cantor_set_perturb}. Set $z_0 :=(x_0, y_0) \in \RR_{\pi}$. For every $0 < r \leq 1$ and 
$\theta_1, \theta_2 \in [0,2\pi)$, with $0< | \theta_1 - \theta_2| < \pi$, we have that 
$\WW(r, \theta_1, \theta_2) \cap (\RR_{\pi} \setminus \OO_f(z_0)) \neq \emptyset$.
\end{lem}

\begin{proof}
First, we observe that $\RR_{\pi} = \QQ_{1,\tu{irr}} \times \QQ_{2,\tu{irr}}$, so the minimal set $\MM$ of $f$ is 
$\MM = \Cl(\QQ_{1,\tu{irr}} \times \QQ_{2,\tu{irr}}) = \QQ_1 \times \QQ_2$, the product of the Cantor sets of the factors
$\varphi$ and $\psi$. Let $B_{\delta_0} \subset \T^2$ be the closed embedded disk centered at $z_0$,
where $\delta_0$ is small enough so that it is contained in the rectangle of width $|J_1|$ and height $|J_2|$ centered at $z_0$.

Relative to the disk $B_{\delta_0}$, let $\WW(r, \theta_1, \theta_2)$ be any wedge, denoted $\WW$ for brevity from now on. 
Let $0 < \nu < \pi$, where $\nu = | \theta_1 - \theta_2 |$ is the angle between the two rays $\rho_{\theta_1}, \rho_{\theta_2}$ 
that bound the wedge, and define $\bar{\rho}:=\rho_{(\theta_1 + \theta_2)/2}$ the bisector of the two rays. Further, let $\nu' \in [0,2\pi)$ 
be the angle between $\bar{\rho}$ and the positive horizontal line through $z_0$. As the vertical and horizontal lines through $z_0$ contain points of $\RR_{\pi} \setminus \OO_f(z_0)$ arbitrarily close to $z_0$, by symmetry, we may assume without loss of generality that 
$0 < \nu' < \pi/2$. Given a point $w = (x,y) \in \bar{\rho} \cap \WW$, let $\ell_h(w), \ell_v(w)$ be the horizontal and vertical straight line through $w$ respectively and define the intercepts $\ell'_h(w) = \ell_h(w) \cap \WW$ and $\ell'_v(w) = \ell_v(w) \cap \WW$, which for $w$ sufficiently close to $z_0$ only pass through $\rho_{\theta_1}$ and $\rho_{\theta_2}$ (and not through the circular arc that cuts off the wedge).

There exist constants $K_h, K_v > 0$, depending only on $\theta_1$ and $\theta_2$, such that $\ell'_h(w) = K_h d(y,y_0)$ and 
$\ell'_v(w) = K_v d(x, x_0)$. Given any $0 < r \leq 1$, choose $w \in \bar{\rho}$ such that $d(w, z_0) < \delta_0 r$.
As the lengths $\ell'_h(w)$, $\ell'_v(w)$ behave linearly, and the lengths of the intervals $|I_{1,k}| \leq C_1 (d(x_0, x_k))^2$ and $|I_{2,t}| \leq C_2 (d(y_0, y_t))^2$ behave quadratically, with $C_1, C_2 >0$ uniform constants, if we choose $w$ sufficiently close to $z_0$, then $w \in R := I_{1,k} \times I_{2,t}$, for some $k ,t \in \Z$, with $\Cl(R)$ properly contained in $\WW$. The cornerpoints of the rectangle $R$ are limit points of $\RR_{\pi}$, and thus also of $\RR_{\pi} \setminus \OO_f(z_0)$, as a neighbourhood of any point $z \in \RR_{\pi}$ contains uncountably many points and only a countable orbit is deleted. Therefore, as $\Cl(R) \subset \WW$, we have that $\WW \cap (\RR_{\pi} \setminus \OO_f(z_0)) \neq \emptyset$, as required.
\end{proof}

\begin{proof}[Proof of Example~\ref{ex_type_III}]
We start with the homeomorphism $f \in \homeo_{\#}(\T^2)$, where $f$ is the product of two Denjoy-counterexamples $\varphi, \psi \in \homeo(\T^1)$, with a Cantor minimal set $\MM_1 = \QQ_1$ and $\MM_2 = \QQ_2$ respectively, semi-conjugate to an irrational translation $\tau$ through $\pi$. As every Cantor set in $\T^1$ can be realized as a minimal set of a Denjoy counterexample, we can choose $\varphi$ and $\psi$ such that the Cantor sets 
$\QQ_i$, with $i=1,2$, with $x_0 \in \QQ_{1, \tu{irr}}$ and $y_0 \in \QQ_{2, \tu{irr}}$ points and corresponding intervals $J_1$ and $J_2$ satisfy the conditions of Lemma~\ref{lem_cantor_set_perturb}. Let $z_0 = (x_0,y_0) \in \RR_{\pi}$ and let $B_{\delta_0} \subset \T^2$ be the closed embedded Euclidean disk with radius $\delta_0>0$ small enough so that $B_{\delta_0}$ is contained in the rectangle of width $|J_1|$ 
and height $|J_2|$ centered at $z_0$. Through the procedure in section~\ref{subsec_constr}, we blow up the orbit $\OO_{f}(z_0)$ to a collection of disks to obtain a homeomorphism $f' \in \homeo_*(\T^2)$ and a continuous $\phi \colon \T^2 \ra \T^2$, such that $\phi \circ f' = f \circ \phi$ and $\phi^{-1}(f^k(z_0)) = \Cl(\Sigma_k)$ is a closed topological disk, where $\gamma_k = \partial \Sigma_k$ is a simple closed curve, for every $k \in \Z$. 

\begin{figure}[h]
\begin{center}
\psfrag{sigma_k}{$\Sigma_k$}
\includegraphics[scale=0.45]{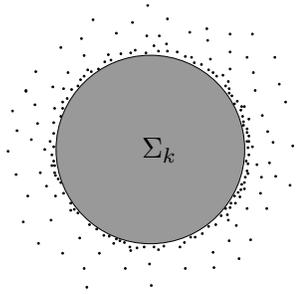}
\caption[A non-trivial extension of a Cantor set]{A non-trivial extension of a Cantor set $\MM'$; Cantor dust accumulating on the boundary $\gamma_k \subset \MM'$ of each disk $\Sigma_k$ and, conversely, every point of the Cantor dust is a limit point of increasingly small disks $\Sigma_k$.}
\label{fig_example_III_1}
\end{center}
\end{figure}

If we denote $\Sigma$ the doubly essential component of the complement of the minimal set $\MM$ of $f$, then $\Sigma' := \phi^{-1}(\Sigma) \subset \T^2$ is open as $\phi$ is continuous. As $\phi : \Sigma' \ra \Sigma$ is one-to-one and continuous, it is a homeomorphism.
Therefore, as $\Sigma$ is doubly essential, take for example the vertical and horizontal essential simple closed curves in $\Sigma$, then these are taken by $\phi^{-1}|_{\Sigma}$ to essential simple closed curves in $\Sigma'$ with the same homotopy type as $\phi$ is homotopic to the identity. Therefore, $\Sigma'$ is doubly essential indeed and thus $\MM'$ is of type III. The other connected components of the complement of $\MM'$ are, by construction, the disks $\Sigma_k$ for $k \in \Z$. 

Now, $\phi^{-1}( \RR_{\pi} \setminus \OO_f(z_0) ) \subset \MM'$ are all singletons, which by Lemma~\ref{lem_struct_min_sets} (2) combined with 
Lemma~\ref{lem_ex_III_cantor_blowup}, accumulate on the boundaries of the disks $\Sigma_k$ to form the non-trivial 
connected components $\gamma_k$, see Figure~\ref{fig_example_III_1}. Thus, as $\MM' = \Cl(\phi^{-1}( \RR_{\pi} \setminus \OO_f(z_0)))$ 
by Lemma~\ref{lem_struct_min_sets} (1), $\gamma_k \subset \MM'$ for every $k \in \Z$, which are the desired non-degenerate components of $\MM'$. 
\end{proof}

\begin{rem}
In the proof of example~\ref{ex_type_III}, we explicitly constructed a semi-conjugacy between the extension of the Cantor set $\MM'$ and the 
original Cantor set $\MM$. Theorem A in essence says that all extensions of Cantor sets are of this form. 
\end{rem}

\section{Concluding remarks}

Let us conclude this paper with the following remarks. Corollary~\ref{cor_limit_sets} states that homeomorphisms that admit a type I or II minimal set have in fact a unique minimal set. 
This poses the natural 

\begin{quest}[Uniqueness of type III minimal sets]\label{open_1}
Let $f \in \homeo_*(\T^2)$ with a minimal set $\MM$ of type III. Is the minimal set $\MM$ unique?
\end{quest}

Note that, by Lemma~\ref{lem_semi_conjugation}, a (possible) counterexample to the above question could not be an element of  $\homeo_{\#}(\T^2)$. Further, it would be interesting to get a completer description 
of the possible topology of extensions of Cantor sets.

\begin{quest}[Topology of type III minimal sets]\label{open_2}
Let $f \in \homeo_*(\T^2)$ with a minimal set $\MM$ of type III. 
\begin{enumerate}
\item[\tu{(i)}] Exactly what continua can arise as a connected component of $\MM$?
\item[\tu{(ii)}] Can $\MM$ have uncountably many non-degenerate connected components? 
Can all components be non-degenerate?
\end{enumerate}
\end{quest}

For example, a classical result by R. Moore~\cite{moore} implies that not all components of $\MM$ can be triodic continua, as the number of connected components of $\MM$ is uncountable.\footnote{A planar continuum $\CC$ is said to be {\em triodic} if there exists a connected 
closed set $\CC_0 \subset \CC$ such that $\CC \setminus \CC_0$ has at least three connected components.} However, it is not entirely clear whether uncountably many components could be for example an arc.

\subsection*{Erratum}

This is a corrected version of the published version~\cite{kwa} of this paper, namely~\cite[Lemma 12]{kwa} claimed that the 
complementary domains of a type II minimal set can not have simultaneously unbounded disks and essential annuli. The claim is certainly true if the unbounded disks are periodic, but it is not known whether it also holds for wandering unbounded disks.

\subsection*{Acknowledgements}
The author thanks Vladimir Markovic and Sebastian van Strien for useful discussions and suggestions that improved the manuscript. Further, the author thanks the referees for their comments and suggestions on the manuscript. The author was supported by 
Marie Curie grant MRTN-CT-2006-035651 (CODY).

\end{document}